\documentclass[11pt,letterpaper]{article}
\pdfoutput=1 
\usepackage[height=8in]{geometry}
\usepackage{amsmath, amssymb, amsthm}
\usepackage{graphicx, color}
\usepackage{array, multirow}
\usepackage[font=small,labelfont=bf]{caption} 
\usepackage{lipsum}
\usepackage{amsfonts}
\usepackage{tabularx}
\usepackage{booktabs}
\usepackage{graphicx}
\usepackage{epstopdf}
\usepackage{authblk}
\usepackage{algorithmic}
\usepackage{tikz}
\ifpdf
  \DeclareGraphicsExtensions{.eps,.pdf,.png,.jpg}
\else
  \DeclareGraphicsExtensions{.eps}
\fi

\usepackage[shortlabels]{enumitem}

\usepackage[sort]{natbib}

\setcitestyle{numbers,square,comma}
\usepackage{hyperref}
\hypersetup{colorlinks=true,
			urlcolor=blue,
			linkcolor=black,
			citecolor=blue,
			bookmarksdepth=paragraph}
\usepackage[nameinlink]{cleveref}
\crefname{equation}{}{}
\crefname{section}{section}{sections}
\crefname{figure}{figure}{figures}
\crefname{table}{table}{tables}
\crefname{example}{example}{examples}
\crefname{proposition}{proposition}{propositions}
\crefname{assumption}{assumption}{assumptions}
\Crefname{section}{Section}{Sections}
\Crefname{figure}{Figure}{Figures}
\Crefname{table}{Table}{Tables}
\Crefname{definition}{Definition}{Definitions}
\Crefname{theorem}{Theorem}{Theorems}
\Crefname{remark}{Remark}{Remarks}
\Crefname{example}{Example}{Examples}
\Crefname{proposition}{Proposition}{Propositions}
\Crefname{assumption}{Assumption}{Assumptions}

\numberwithin{equation}{section}

\newtheorem{theorem}{Theorem}[section]
\newtheorem{lemma}{Lemma}[section]
\newtheorem{corollary}{Corollary}[section]

\theoremstyle{definition}

\newtheorem{remark}{Remark}[section]
\newtheorem{assumption}{Assumption}[section]
\newtheorem{example}{Example}[section]

\title{Auxiliary functions as Koopman observables: Data-driven analysis of dynamical systems via polynomial optimization}
\author[1]{Jason J. Bramburger}
\author[2]{Giovanni Fantuzzi}
\affil[1]{Department of Mathematics and Statistics, Concordia University,\newline Montr\'eal, QC, Canada}
\affil[2]{Department of Mathematics, Friedrich--Alexander--Universit\"at Erlangen--N\"urnberg, Erlangen, Germany}

\newcommand{\R}{\mathbb{R}} 
\newcommand{\bS}{\mathbb{S}} 
\renewcommand{\vec}[1]{\boldsymbol{#1}} 
\newcommand{\proj}[2]{\mathcal{P}_{#1}^{#2}} 

\newcommand{\dictionarySpan}{\linspan\vec{\psi}}
\newcommand{\afSpan}{\linspan\vec{\phi}}
\newcommand{\timeSet}{\mathbb{T}}
\newcommand{\stateSet}{\mathbb{X}}
\newcommand{\timestep}{\tau}

\newcommand{\koopman}{\mathcal{K}}

\newcommand{\Lie}{\mathcal{L}} 

\DeclareMathOperator{\linspan}{span}
\DeclareMathOperator*{\argmin}{argmin}
\DeclareMathOperator*{\Argmin}{Argmin}
\DeclareMathOperator{\supp}{supp}
\DeclareMathOperator{\rank}{rank}

\DeclareMathOperator{\trace}{trace}

\begin{document}

\maketitle

\begin{abstract}
We present a flexible data-driven method for dynamical system analysis that does not require explicit model discovery. The method is rooted in well-established techniques for approximating the Koopman operator from data and is implemented as a semidefinite program that can be solved numerically. Furthermore, the method is agnostic of whether data is generated through a deterministic or stochastic process, so its implementation requires no prior adjustments by the user to accommodate these different scenarios. Rigorous convergence results justify the applicability of the method, while also extending and uniting similar results from across the literature. Examples on discovering Lyapunov functions, performing ergodic optimization, and bounding extrema over attractors for both deterministic and stochastic dynamics exemplify these convergence results and demonstrate the performance of the method. 
\end{abstract}

\section{Introduction}

In his now famous work~\cite{koopman1931hamiltonian}, Koopman presented an equivalent linear formulation of nonlinear systems through what is now called the Koopman operator. This linear description of genuinely nonlinear systems comes at the expense of lifting the dynamics to an infinite-dimensional Banach space of functions called {\em observables}. Nevertheless, the Koopman operator has become increasingly popular in recent years to address two broad classes of problems in nonlinear dynamics.

The first class of problems is concerned with providing a geometric interpretation for observed properties of dynamical systems. This is primarily achieved by extracting eigenvalues and eigenfunctions of the Koopman operator, so that nonlinear systems can be understood through well-developed techniques for linear systems. For example, Koopman eigenfunction expansions constitute a space-time separation of variables that can be used for forecasting and producing reduced order models, while the eigenfunctions themselves are coherent structures present in the system \cite{brunton2022modern}.
One can also exploit the linearity of the Koopman operator to apply linear control methods to nonlinear system \cite{brunton2016koopman,mamakoukas2019local,abraham2019active,korda2018linear,peitz2019koopman,kaiser2021data} and to quantify uncertainty \cite{mezic2008uncertainty,mezic2004uncertainty}. 

What makes Koopman theory particularly attractive for these problems is the availability of machine learning techniques that can estimate the Koopman operator and/or its spectrum directly from dynamic data (see the review~\cite{brunton2022modern} and the references therein). In this work we focus on one such data-driven technique, called \emph{extended dynamic mode decomposition} (EDMD), that seeks to approximate the action of the Koopman operator on the span of finitely many observables (a \emph{dictionary})~\cite{eDMD}. The result is a finite-dimensional matrix whose spectrum can be numerically extracted and used to approximate that of the true Koopman operator. Importantly, the  good practical performance of EDMD is rigorously justified by convergence guarantees as the amount of data increases~\cite{eDMD,eDMDconv2} and as the dictionary grows~\cite{eDMDconv}.

The second class of problems where the Koopman operator has enjoyed tremendous success is model-based system analysis. In this case, one appeals to Koopman theory to prove \emph{a priori} statements about dynamical systems by finding {\em auxiliary functions} whose derivatives along trajectories (called \textit{Lie derivatives}) satisfy pointwise inequalities implying the desired result. A familiar example of auxiliary functions are the Lyapunov functions used in stability analysis~\cite{Lyapunov1892stability}, which attain a global minimum at an equilibrium and decay monotonically along all other trajectories. Other types of auxiliary functions can be used to bound infinite-time averages~\cite{goluskin2018bounding,tobasco2018optimal}, stochastic expectations~\cite{Fantuzzi2016siads,Kuntz2016bounding}, and extreme values along trajectories~\cite{fantuzzi2020bounding} or over attractors~\cite{goluskin2020bounding}; approximate reachable sets~\cite{Korda2014convex,Magron2019semidefinite}, basins of attraction~\cite{Tan2006stability,Korda2013inner,Henrion2014convex,Valmorbida2017region}, attractors~\cite{Jones2019,Schlosser2022,Schlosser2021}, and invariant sets~\cite{Prajna2006barrier,bramburger2020minimum,Parker2021pendulum}; estimate system parameters and propagate uncertainty~\cite{Streif2013,Streif2014,Miller2021uncertain,covella2022uncertainty}; solve optimal control and optimal stopping problems~\cite{Hernandez1996linear,Cho2002linear,Lasserre2008nonlinear,Henrion2008nonlinear}. 
The key observation in all of these applications is that Lie derivatives of auxiliary functions can be accessed via the Koopman generator without knowledge of system trajectories. Moreover, and crucially for practical applications, the linearity of the Koopman generator ensures that constraints on auxiliary functions are convex. This makes it possible to optimize auxiliary functions computationally if the dynamics are governed by known polynomial equations. Briefly, if one searches for polynomial auxiliary functions with tunable coefficients, then the associated constraints are polynomial inequalities in the time and state variables that (by the linearity of the Koopman generator) depend linearly on the tunable auxiliary function coefficients. Such inequalities are NP-hard to verify in general~\cite{polyNP}, but can be strengthened by requiring that any polynomial to be non-negative is a sum of squares (SOS). These SOS constraints can be reformulated as semidefinite programs~\cite{NesterovSOS,lasserre2001global,Parrilo2003}, for which efficient software packages exist. 

In this work, we unite the data-driven approximation of the Koopman operator with the auxiliary function perspective to show how to extract important information about the system directly from data. This is possible because the Lie derivative operator entering the constraints on auxiliary functions is the generator of the Koopman operator, which can be approximated on observable functions using EDMD. Thus, one may view auxiliary functions as special Koopman observables and approximate them by replacing the exact Lie derivative with data-driven approximations built using EDMD. We demonstrate this in the context of stability analysis, bounding time averages, and bounding extreme values on attractors, but the same ideas apply to all auxiliary function frameworks listed above. In fact, similar ideas have already been used to construct Lyapunov functions~\cite{moyalan2022data,deka2022koopman} and to approximate basins of attraction \cite{zheng2022data} or controlled invariant sets~\cite{KordaDataInvariantSet} from data. Here, we generalize these ideas by going beyond stability analysis and by considering a broad class of stochastic systems that are not necessarily governed by polynomial equations. For this class of systems, we give a complete picture of how EDMD can be used to build approximate Lie derivatives from data. We also show how approximate auxiliary functions for many applications beyond stability analysis can be discovered from data through polynomial optimization to make statements about the underlying dynamical system.

In addition to being applicable to all existing auxiliary function frameworks for dynamical system analysis, our data-driven approach has two major strengths. First, it can be applied to a broad class of deterministic or stochastic processes (again, not necessarily governed by polynomial equations) that evolve in either continuous or discrete time. In fact, no adjustments are needed to implement our method on deterministic versus stochastic data, so it can easily be transferred between applications. This follows from theoretical analysis in \cref{sec:Theory}, which shows that EDMD approximates the correct expression for the Lie derivative irrespective of whether the dynamical process underlying the data is deterministic or stochastic.
The second strength of our approach is that, since it directly discovers approximate Lie derivatives, it does not require the identification of a model for the dynamics. As a result, we are able to approximate auxiliary functions from data even when it would be difficult to first perform model discovery using techniques like SINDy~\cite{SINDy,BramPoincare,WeakSINDy1,WeakSINDy2,McCalla,kaheman2020sindy,rudy2017data,kaptanoglu2021promoting} (see \cref{ss:stoch-logistic-results} for an example in the context of stochastic dynamics). Bypassing the model idenfication step also means that approximate auxiliary functions provide system-level information that apply to the data. In contrast, existing model discovery techniques such as SINDy often prioritize the interpretability of the discovered model over its accuracy, so auxiliary functions identified from a discovered model could have no relationship to the original data due to the inaccuracy built into the model identification step.

In summary, our contributions in this work are twofold. First, for a broad class of stochastic processes defined precisely in \cref{s:setup}, we present a method to discover approximate auxiliary functions from data using polynomial optimization. Specifically, we show how to approximate Lie derivative from data using EDMD (\cref{sec:DataAux}) and how to combine these approximations with tools for polynomial optimization (\cref{sec:sdps}). Second, we provide rigorous convergence results on the EDMD-based estimation of the Lie derivative from data (\cref{sec:Theory}). These extend previous convergence results for EDMD~\cite{eDMD,eDMDconv,eDMDconv2} to a more general class of stochastic processes, and hold under weaker assumptions. We also discuss in \cref{ss:gEDMD}  the possibility of replacing EDMD with \emph{generator EDMD} (gEDMD) \cite{Klus2020}, revealing that the two approaches may give very different results even if the amount of data and the data sampling rate become infinite.

The performance of our approach and our theoretical convergence results are illustrated in \cref{sec:FirstEx} by a range of numerical and analytical examples. Precisely, we demonstrate that one can identify Lyapunov functions, bound long-time deterministic and stochastic averages, and bound pointwise extrema over attractors using only data gathered from an underlying system. These three particular applications are reviewed in \cref{sec:AuxFunctions} and serve as motivating examples for the use of  auxiliary functions for dynamical systems analysis. Howevever, we emphasize once again that our data-driven approach is not limited to these three applications, but can be applied to any of the auxiliary function frameworks listed earlier in this introduction. Concluding remarks and an outline of potential avenues for future work are offered in \cref{sec:Conclusion}.
\section{Motivating examples}
\label{sec:AuxFunctions}

To set the scene, we begin by reviewing how Lie derivatives of auxiliary functions enable one to perform stability analysis, bound infinite-time averages, and bound extreme values of observables on attractors. 
To ease the discussion, we focus here on deterministic continuous-time processes governed by a nonlinear ODE $\dot{X_t} = f(X_t)$, where the state $X_t$ is a point in $\R^d$ at each time $t$. In this case, the Lie derivative of a continuously differentiable function $\varphi:\R^d \to \R$ is its derivative along ODE solutions,
\begin{equation*}
	\Lie \varphi(x) := f(x) \cdot \nabla \varphi(x).
\end{equation*}
However, the three auxiliary function frameworks reviewed below can be applied to discrete-time deterministic dynamics and to stochastic dynamics simply by replacing this definition of $\Lie \varphi$ with those given in \cref{s:setup}.

\subsection{Global and local stability}
Let $X_t \in \R^d $ be the solution of the ODE $\dot{X_t} = f(X_t)$ at time $t$ and assume $f$ satisfies $f(0) = 0$, so the point $X=0$ is an equilibrium. Lyapunov \cite{Lyapunov1892stability} showed that the equilibrium point $X = 0$ is globally stable if there exists a continuously differentiable function $V:\mathbb{R}^d\to\R$ satisfying
\begin{subequations}
	\begin{align}
		\label{e:lyap:positivity}
		&&&&&&V(x) &\geq 0 &&\forall x \in \mathbb{R}^d,&&&&&&\\
		\label{e:lyap:decay}
		&&&&&&\Lie V(x) &\leq 0 &&\forall x \in \mathbb{R}^d,&&&&&&\\
		\label{e:lyap:coercivity}
		&&&&&&V(x) &\to +\infty &&\text{as }\|x\| \to +\infty.&&&&&&
	\end{align}
\end{subequations}
In particular, one has global asymptotic stability if $V(0)=0$ and the inequalities in \cref{e:lyap:positivity,e:lyap:decay} are strict whenever $x \neq 0$. Local (asymptotic) stability can be proved by imposing \cref{e:lyap:positivity,e:lyap:decay} only in a neighbourhood $S$ of the equilibrium point, which implies the largest sublevel set of $V$ included in $S$ is positively invariant \cite[\S4.8]{Khalil2002}.

\subsection{Ergodic optimization}
\label{sec:ergodic-optimization}
Auxiliary functions can be used to estimate long-time averages, a problem at the heart of ergodic theory. Given a trajectory $X_t$ of an ODE $\dot{X_t} = f(X_t)$ with initial condition $X_0$, the long-time average of a continuous function $g:\R^d \to \R$ is defined as
\begin{equation}\label{ErgodicAvg}
	\overline{g}(X_0) = \limsup_{T \to \infty} \frac{1}{T}\int_0^T {g}(X_t)\mathrm{d}t.
\end{equation}

We seek to compute the largest possible long-time average among trajectories starting from a compact, forward-invariant set $S \subset \mathbb{R}^d$,
\begin{equation}
	\overline{g}^* := \sup_{X_0 \in S} \overline{g}(X_0).
\end{equation}
(The minimal time average can be deduced by negating upper bound on the maximal time average of $-g$.)
The value $\overline{g}^*$ can be determined using auxiliary functions with no need to determine explicit optimal trajectories. Precisely, it was shown in~\cite{tobasco2018optimal} that
\begin{equation}\label{AvgBnd}
	\overline{g}^*
	= \inf_{\substack{U \in \R\\V\in C^1(B,\R)}} \left\{
	U:\; U - {g}(x) + \Lie V(x) \geq 0 \quad \forall x \in B
	\right\}.
\end{equation}
In particular, any feasible auxiliary function $V$ and corresponding constant $U$ yield the upper bound $\overline{g}^* \leq U$.

\subsection{Attractor bounds}
\label{sec:attractor-bounds}

We again consider an ODE $\dot{X_t} = f(X_t)$ and a continuous function $g:\R^d \to \R$. Instead of bounding the extremal time-averaged behaviour, we can seek to determine the extremal value that $g$ can attain over the ODE's global attractor $\mathcal{A}$. Precisely, we seek to bound the maximal value
\begin{equation}
	g_\mathcal{A}^+ := \max_{x \in \mathcal{A}} g(x).
\end{equation}  
(Lower bounds on $g$ over $\mathcal{A}$ can be deduced by negating upper bounds on $-g$.)
As demonstrated in \cite{goluskin2020bounding}, if the unknown attractor $\mathcal{A}$ is contained in a known positively invariant set $S \subset \R^d$, then $g_\mathcal{A}^+$ can be bounded from above using auxiliary functions by
\begin{align} \label{AttractorBound}
	g^+_\mathcal{A} \leq \inf_{\substack{\lambda > 0\\ V \in C^1(S)}}
	\big\{ U:\;
	U - V(x) - \lambda \mathcal{L} V(x) &\geq 0\quad \forall x\in S
	\nonumber
	\\[-3ex]
	\text{and } V(x) - g(x) &\geq 0\quad \forall x\in S \big\}.
\end{align}
The minimization problem on the right-hand side is convex for every fixed $\lambda$. Generalizations that evaluate $g^+_\mathcal{A}$ exactly also exist \cite{Schlosser2021}, but they introduce additional auxiliary functions and will not be considered here for simplicity.
\section{A class of dynamical systems}
\label{s:setup}

Having reviewed three particular applications of auxiliary functions, we now switch gear and introduce a general class of stochastic Markov processes whose dynamics can be studied using auxiliary functions and their Lie derivatives. As explained at the end of the section, this class includes processes governed by stochastic differential equations, stochastic maps, and their deterministic counterparts. Indeed, observe that deterministic processes may be viewed as stochastic ones whose state at time $t$ is determined almost surely given the state at any previous time $s<t$.

\subsection{General stochastic framework}\label{ss:setup:general}
Let $X_t$ denote the state at time $t$ of a stochastic process on a probability space $(\Omega, \mathcal{F}, \pi)$, which evolves in a subset $\stateSet$ of a Banach space over either the continuous time set $\timeSet = \R_+$ or the discrete time set $\timeSet = \mathbb{N}$.  
We write $\mathbb{E}[\varphi(s,X_{s})|X_{t}=x]$ for the expected value of $\varphi(s,X_{s})$ at time $s\geq t$ given that $X_{t}=x$, with the understanding that $\mathbb{E}[\varphi(s,X_{s})\,|\,X_{t}=x] = \varphi(s,X_{s})$ for deterministic dynamics.
The \emph{generator} of the process is the linear operator $\Lie$ defined on the space $C_{b}(\timeSet \times \stateSet)$ of bounded continuous functions on $\timeSet \times \stateSet$ via
\begin{equation*}
	\Lie\varphi(t,x):=\mathbb{E}[\varphi(t+1,X_{t+1})\,|\,X_{t}=x]-\varphi(t,x)
\end{equation*}
in discrete time and by
\begin{equation*}
	\Lie\varphi(t,x)=\lim_{\timestep\to0^{+}}\frac{\mathbb{E}[\varphi(t+\timestep,X_{t+\timestep})\,|\,X_{t}=x]-\varphi(t,x)}{\timestep}
\end{equation*}
in continuous time, provided the limit exists uniformly on $\timeSet\times\stateSet$. We write $\mathcal{D}(\Lie)$ for the domain of $\Lie$ and we call $\Lie\varphi$ the \emph{Lie derivative} of $\varphi$ since, for deterministic processes, $\Lie\varphi$ gives simply the difference (in discrete time) or derivative (in continuous time) along trajectories of the process. Note that $\varphi$ is assumed bounded to ensure that expectations are finite. For deterministic processes, however, Lie derivatives are well defined for all sufficiently smooth functions even if they are unbounded.

We will restrict our attention to stochastic processes that are Markov and solve the so-called \emph{martingale problem} for their generators $\Lie$. This means that, for all times $s\geq t$ and all $\varphi$ in the domain of $\Lie$, we have
\begin{subequations}\label{eq:martingale-all}
\begin{equation}\label{eq:martingale-discrete}
	\mathbb{E}\left[\varphi(s,X_{s})\;|\;X_{t}=x\right]=\varphi(t,x)+\mathbb{E}\left[\sum_{\tau=t}^{s-1} \Lie\varphi(\tau,X_{\tau})\;|\;X_{t}=x\right]
\end{equation}
in discrete time and
\begin{equation}\label{eq:martingale-continuous}
	\mathbb{E}\left[\varphi(s,X_{s})\;|\;X_{t}=x\right]
	= \varphi(t,x) +
	\mathbb{E}\left[
		\int_{t}^{s}\Lie\varphi(\tau,X_{\tau})\,d\tau\;|\;X_{t}=x
		\right]
\end{equation}
\end{subequations}
in continuous time. A detailed treatment of martingale problems and their use in characterizing stochastic processes can be found in \cite{EthierKurtz1986}. 

Given a Markov process in the class just described and a positive timestep $\timestep \in \timeSet$, one can define a linear operator $\koopman^\timestep$ on $C_b(\timeSet \times\stateSet)$, sometimes called the \emph{stochastic Koopman operator}~\cite{Klus2020,Crnjaric-Zic2020,Wanner2022}, that maps a function $\varphi$ to
\begin{equation*}
	\koopman^\timestep\varphi(t,x) := \mathbb{E}[\varphi(t+\timestep,X_{t+\timestep})|X_{t}=x].
\end{equation*}
One can use the relevant condition in \cref{eq:martingale-all} and the Markov property to check that the family $\{\koopman^\timestep: \timestep \in \timeSet \}$ of Koopman operators is a one-parameter contraction semigroup on $C_b(\timeSet \times \stateSet)$ for the $L^\infty$ norm. The generator of the semigroup, of course, is $\Lie$.

\subsection{Classical examples}\label{subs:classical-examples}
The general framework introduced above includes processes $X_{t}$ that are governed by deterministic maps, stochastic maps, ODEs, and stochastic differential equations. We briefly explain this here, giving the corresponding expression for the Lie derivatives.

\begin{example}[Deterministic maps]
	Let $\{X_{t}\}_{t\in\mathbb{N}}$ be a discrete-time process governed by the deterministic map $X_{t+1}=f(t,X_{t})$. Then, condition \cref{eq:martingale-discrete} holds for any continuous function $\varphi$ with $$\Lie\varphi(t,x):=\varphi(t+1,f(t,x))-\varphi(t,x).$$
\end{example}

\begin{example}[Ordinary differential equations]
	Set $\stateSet=\mathbb{R}^{d}$, $\timeSet = \mathbb{R}_+$, and let $X_{t}$ solve the ODE $\dot{X_{t}}=f(t,X_{t})$ for some locally Lipschitz continuous function $f:\mathbb{R}_{+}\times\mathbb{R}^{d}\to\mathbb{R}^{d}$.
	Then, condition \cref{eq:martingale-continuous} holds for any continuously differentiable function $\varphi$ with $$\Lie\varphi := \partial_{t}\varphi+f\cdot\nabla_{x}\varphi.$$
\end{example}

\begin{example}[Stochastic maps]
	Let $\{X_{t}\}_{t\in\mathbb{N}}$ be a discrete-time stochastic process governed by the random map $X_{t+1}=f(\omega(t),t,X_{t})$,	where the function $\omega\mapsto f(\omega,\cdot,\cdot)$ is a random variable from some probability space $(\Omega,\mathcal{F},\pi)$ into the space of maps from $\timeSet \times \stateSet$ to $\stateSet$.
	Equivalently, the value of $X_{t+1}$ is sampled randomly from some stochastic kernel $\nu_{t,x}$, meaning a probability measure on $\stateSet$ that depends on the time $t$ and on the value $x$ taken by $X_{t}$. Condition \cref{eq:martingale-discrete} holds for any bounded continuous function $\varphi$ with
	\begin{align*}
		\Lie\varphi(t,x) 
		&:=\int_{\Omega}\varphi(t+1,f(\omega,t,x))d\pi(\omega)-\varphi(t,x)\\
		&=\int_{\stateSet}\varphi(t+1,y)d\nu_{t,x}(y)-\varphi(t,x).
	\end{align*}
\end{example}

\begin{example}[Stochastic differential equations]
	Set $\stateSet=\mathbb{R}^{d}$, $\timeSet = \mathbb{R}_+$, and let $X_{t}$ solve the stochastic differential equation $dX_{t}=f(t,X_{t})dt+g(t,X_{t})dW(t)$ for some locally Lipschitz functions
	$f:\mathbb{R}_{+}\times\mathbb{R}^{d}\to\mathbb{R}^{d}$ and \
	$g:\mathbb{R}_{+}\times\mathbb{R}^{d}\mapsto\mathbb{R}^{d\times k}$,
	where $dW(t)$ is a $k$-dimensional Brownian process. Dynkin's formula shows that \cref{eq:martingale-continuous} holds for functions $\varphi$ that are twice continuously differentiable and bounded with
	\begin{equation*}
		\Lie\varphi
		:=
		\partial_{t}\varphi
		+ f\cdot\nabla_{x}\varphi
		+ \frac{1}{2}
		\left\langle g g^{\top}, \nabla_{x}^{2}\varphi \right\rangle,
	\end{equation*}
	where $\nabla_{x}^{2}\varphi$ is the Hessian of $\varphi$ with respect to the $x$ variable and $\langle A,B\rangle=\sum_{i,j}A_{ij}B_{ij}$.
\end{example}
\section{Data-driven approximation of Lie derivatives}
\label{sec:DataAux}

As demonstrated by the examples in \cref{sec:AuxFunctions}, constructing auxiliary functions for dynamical system analysis requires knowledge of their Lie derivatives. In \cref{sec:EDMD} below, we describe how to use EDMD \cite{eDMD} to build accurate Lie derivative approximations from data for the stochastic processes introduced in \cref{ss:setup:general}. We then establish rigorous convergence guarantees for these approximations in the limits of infinite data, infinite data sampling rate, and infinite dictionaries. In particular:
\begin{enumerate}
	\item \Cref{thm:inf-data-limit} in \cref{s:InfData} generalizes infinite-data convergence results for dynamics governed by particular equations proven in \cite{eDMDconv,eDMDconv2,Crnjaric-Zic2020,Wanner2022} to general Markov processes solving the martingale problem for their generator (see \cref{s:setup}). Our proof also removes invertibility assumptions required in these previous works to pass to the limit along certain matrix sequences.
	\item \Cref{thm:h-to-zero} in \cref{s:InfSample}, coupled with the discussion of \cref{ss:gEDMD}, elucidates the link between EDMD with infinite sampling rate ($\timestep \to 0$) and generator EDMD (gEDMD). Contrary to what one might expect, we show that approximations via EDMD and gEDMD need not coincide as the data sampling rate increases, and we provide conditions under which they do (see \cref{thm:h-to-zero} and its \cref{cor:h-to-zero-pointwise-conv}).
	\item \Cref{thm:conv-m-infty-poitwise} in \cref{s:InfDictionary} provides conditions for the {\em pointwise} convergence of approximate Lie derivatives to exact ones, whereas existing results only guarantee convergence in a suitable Lebesgue norm. 
\end{enumerate}
More details on these extensions are provided throughout the section. Readers who are interested in the practical implementation of our methods may initially focus only on \cref{sec:EDMD} and then proceed directly to \cref{sec:sdps}, where we combine approximate Lie derivatives with polynomial optimization to perform system analysis on data.

\subsection{The EDMD method}
\label{sec:EDMD}

Let $\phi_1,\ldots, \phi_\ell$ and $\psi_1, \ldots, \psi_m$ be two finite dictionaries in $C_b(\timeSet \times \stateSet)$ whose elements are referred to as \emph{observables}. We set
\begin{equation*}
    \vec{\phi} :=\begin{pmatrix}\phi_{1}\\ \vdots \\ \phi_{\ell}\end{pmatrix},
    \qquad
    \vec{\psi} := \begin{pmatrix}\psi_{1} \\ \vdots \\ \psi_{m} \end{pmatrix},
\end{equation*}
and write $\afSpan$ (resp. $\dictionarySpan$)
for the linear span of $\phi_{1}$, $\ldots$, $\phi_{\ell}$ (resp. $\psi_{1}$, $\ldots$, $\psi_{m}$).  

Let there be given $n$ `data snapshots', where $x_{i} = X_{t_{i}}$ and $y_{i}=X_{t_{i}+\timestep}$ for a fixed time increment $\timestep>0$. The EDMD framework constructs an \emph{approximate Koopman operator}, $\koopman_{mn}^{\timestep}:\afSpan\to\dictionarySpan$, that approximates the action of the exact Koopman operator $\koopman^\timestep$ on $\afSpan$ using linear combinations of functions in $\dictionarySpan$. To build this operator, define the matrices
\begin{gather*}
    \Psi_{n}:=
    \begin{bmatrix}
        | &  & |\\
        \vec{\psi}(t_{1},x_{1}) & \cdots & \vec{\psi}(t_{n},x_{n})\\
        | &  & |
    \end{bmatrix} \in\mathbb{R}^{m \times n},
    \\[1ex]
    \Phi_{n}^{\timestep}:=
    \begin{bmatrix}
        | &  & |\\
        \vec{\phi}(t_{1}+\timestep,y_{1}) & \cdots & \vec{\phi}(t_{n}+\timestep,y_{n})\\
        | &  & |
    \end{bmatrix} \in \mathbb{R}^{\ell\times n},
\end{gather*}
and set 
\begin{equation}\label{e:Kmnh-def}
    K_{mn}^{\timestep}:=\Phi_{n}^{\timestep}\Psi_{n}^{\dagger} 
    = \left(\Phi_{n}^{\timestep}\Psi_{n}\right) \left( \Psi_{n}\Psi_{n}^\top\right)^{\dagger},
\end{equation} 
where the superscript ${\dagger}$ denotes the Moore--Penrose pseudoinverse. Observe that $K_{mn}^{\timestep}$ minimizes the Frobenius norm $\|\Phi_{n}^{\timestep}-K\Psi_{n}\|_{F}$ over all $\ell\times m$ matrices $K$, and has the smallest Frobenius norm among all optimizers. (These are not unique unless $\Psi_{n}\Psi_{n}^{\top}$ is invertible, which we do not assume.)
The approximate Koopman operator $\koopman_{mn}^{\timestep}$ then acts on a function $\varphi=\vec{c}\cdot\vec{\phi}$ with $\vec{c}\in\mathbb{R}^{\ell}$ by
\begin{equation}\label{e:approx-koopman}
\koopman_{mn}^{\timestep}\varphi:=\vec{c}\cdot K_{mn}^{\timestep}\vec{\psi}.
\end{equation}

Once an approximate Koopman operator is available, it is almost immediate to approximate Lie derivatives. To this end, we introduce the following assumption.
\begin{assumption}\label{ass:dictionary}
    The dictionaries $\vec{\phi}$, $\vec{\psi}$ satisfy
    $\afSpan\subseteq\mathcal{D}(\Lie)$ and $\afSpan\subseteq\dictionarySpan$.
\end{assumption}
Assuming that $\afSpan\subseteq\mathcal{D}(\Lie)$ ensures functions in $\afSpan$ have well-defined Lie derivatives. This assumption is easily satisfied for the four classes of dynamical processes in \cref{subs:classical-examples}: for ODE dynamics, for instance, $\vec{\phi}$ can be any dictionary of continuously differentiable functions. In the data-driven setting we have in mind, however, one does not know $\mathcal{L}$ or its domain, so one must choose a reasonable dictionary $\vec{\phi}$ bearing in mind that any results one derives are conditional on \cref{ass:dictionary}. This difficulty, of course, is common to all data-driven Koopman approximation methods.

The second part of \cref{ass:dictionary}, instead, can always be satisfied through an appropriate choice of $\vec{\psi}$ and implies that
\begin{equation}
    \vec{\phi}=\Theta_{m}\vec{\psi}
    \label{eq:phi-to-psi}
\end{equation}
for some $\ell\times m$ matrix $\Theta_{m}$.
One can then build an approximate Lie derivative operator $\Lie_{mn}^{\timestep}$ from $\afSpan$ into $\dictionarySpan$ simply by defining, for every $\varphi=\vec{c}\cdot\vec{\phi}$,
\begin{equation}
    \Lie_{mn}^{\timestep}\varphi=\vec{c}\cdot L_{mn}^{\timestep}\vec{\psi}
    \qquad\text{where}\qquad
    L_{mn}^{\timestep}:=\frac{K_{mn}^{\timestep}-\Theta_{m}}{\timestep}.
	\label{e:Lmnh-def}
\end{equation}
This is essentially a finite-difference approximation of $\Lie\varphi$, since using \cref{eq:phi-to-psi} one finds
\begin{equation*}
\Lie_{mn}^{\timestep}\varphi
=\frac{\vec{c}\cdot K_{mn}^{\timestep}\vec{\psi}-\vec{c}\cdot\vec{\phi}}{\timestep}
=\frac{\koopman_{mn}^{\timestep}\varphi-\varphi}{\timestep}
\approx\frac{\koopman^{\timestep}\varphi-\varphi}{\timestep}
\approx\Lie\varphi.
\end{equation*}
Analysis in the next subsection rigorously justifies these heuristic approximations. Precisely, under reasonable assumptions on the data snapshots, we prove that $\Lie^\timestep_{mn}\varphi$ converges to $\Lie \varphi$ in a suitable norm as $n\to\infty$, $\timestep\to0$, and $m\to\infty$.

\begin{remark}
    While EDMD is typically applied with $\vec{\psi}=\vec{\phi}$, using two different dictionaries may produce more accurate results. This is especially true if, as is often the case, $\afSpan$ is not closed under the action of the Koopman generator $\Lie$. Examples are given in \cref{sec:FirstEx}. This observation also underpins EDMD-based system discovery methods such as SINDy~\cite{SINDy,BramPoincare,WeakSINDy1,WeakSINDy2,McCalla,kaheman2020sindy,rudy2017data,kaptanoglu2021promoting}, where $\vec{\phi}=(X_1,\ldots,X_d)$ lists only the state variables while $\vec{\psi}$ is a rich dictionary of nonlinear functions.
\end{remark}

\begin{remark}
    Contrary to classical implementations of EDMD, we use dictionaries $\vec{\phi}$ and $\vec{\psi}$ with explicit time dependence. This extension is necessary to ensure that approximate Lie derivatives can be employed to implement auxiliary function frameworks for non-autonomous dynamics or for finite time horizons. Indeed, such cases typically require auxiliary functions with explicit dependence on the time variable (see, e.g., \cite{Lasserre2008nonlinear,Streif2013,Korda2014convex,fantuzzi2020bounding,Miller2021peak,Schlosser2021,Schlosser2022,covella2022uncertainty}).
\end{remark}

\subsection{Convergence results}\label{sec:Theory}


We now prove that the approximate Lie derivative $\Lie_{mn}^\timestep \varphi$  in \cref{e:Lmnh-def} converges to the exact Lie derivative $\Lie\varphi$ in the limits of infinite data ($n\to\infty$), infinite data sampling rate ($\timestep\to0$), and infinite EDMD dictionary $\vec{\psi}$ ($m\to\infty$). Our results extend known statements (see, e.g., \cite{eDMDconv,eDMDconv2,Klus2020,Crnjaric-Zic2020,Wanner2022}) to a broader class of stochastic processes and apply under weaker assumptions. We focus on continuous-time processes, but equivalent results for discrete-time processes can be recovered by setting $\timestep=1$ in what follows and ignoring results about the $\timestep\to 0$ limit.

\subsubsection{Assumptions on the data sampling method}
\label{ss:sampling-assumptions}
We assume the data snapshots $(t_{i},x_{i},y_{i})_{i=1}^n$ satisfy $x_i = X_{t_i}$ and $y_i = X_{t_i + \timestep}$. If the distribution of the random variable $X_{t+\timestep}$ given $X_{t}=x$ is described by a probability measure $\nu_{t,x}$ on $\stateSet$, each $y_i$ is a random variable with distribution $\nu_{t_i, x_i}$. We further assume that each pair $(t_i, x_i)$ is sampled from a probabilty measure $\mu$ on $\timeSet\times\stateSet$. Then, the data snapshots $(t_{i},x_{i},y_{i})$ are random variables whose joint distribution is the probability measure $\rho$ on $\timeSet\times\stateSet\times\stateSet$ defined for every Borel subsets $E_t\subset \timeSet$ and $E_x,E_y \subset \stateSet$ via
\begin{equation*}
	\rho(E_t,E_x,E_y) := \int_{E_t\times E_x} \nu_{t,x}(E_y)\;d\mu(t,x).
\end{equation*}

Our analysis will rely on the following assumption, where the quantifier \emph{almost surely} means that a statement holds for almost all sequences $\{(t_i,x_i,y_i)_{i=1}^n \}_{n\geq 1}$ of data snapshots generated by the sampling strategy.

\begin{assumption}\label{ass:weak-convergence}
    For every function $g\in C_{b}(\timeSet\times\stateSet\times\stateSet)$, there almost surely holds
    \begin{equation}
    \lim_{n\to\infty} \frac{1}{n}\sum_{i=1}^{n}g(t_{i},x_{i},y_{i})
	=
    \int g(t,x,y)\,d\rho(t,x,y).
    \label{eq:weak-convergence}
    \end{equation} 
\end{assumption}

This assumption is standard in the analysis of EDMD and its variations, and can be ensured in two ways. If $\rho$ is an ergodic measure, one can collect the snapshots from a trajectory of the dynamical system. Alternatively, if $\stateSet$ is a separable Banach space, one can sample the snapshots independently from $\rho$~\cite{Varadarajan1958}. In the latter case, the theory of Monte Carlo integration (see, e.g., \cite{EvansSwartz2000}) ensures convergence at a rate of $1/\sqrt{n}$. In the former case, instead, no general convergence rate can be stated because ergodic averages can converge arbitrarily slowly \cite{Krengel1978}.

\subsubsection{Orthogonal projections}
\label{ss:orth-proj}

Recall from \cref{s:setup} that the Lie derivative operator $\Lie$ is the generator of the Koopman semigroup on $C_b(\timeSet \times \stateSet)$, which is equipped with the topology of uniform convergence. 
For our analysis, however, it will be convenient to view $C_b(\timeSet \times \stateSet)$ and its subspaces $\afSpan$ and $\dictionarySpan$ as subspaces of $L_{\mu}^{2}(\timeSet\times\stateSet)$, the Lebesgue space of functions $\varphi:\timeSet\times\stateSet \to \mathbb{R}$ that are square-integrable with respect to the probability measure $\mu$ from \cref{ss:sampling-assumptions}. The norm on this space is
\begin{equation*}
\|\varphi\|_{L_{\mu}^{2}}:= \bigg(\int_{\timeSet\times\stateSet}|\varphi(t,x)|^{2}d\mu(t,x)\bigg)^\frac{1}{2}.
\end{equation*}

We will work extensively with the projection of functions in $L^2_\mu(\timeSet\times\stateSet)$ onto $\dictionarySpan$. For every $f \in L^2_\mu(\timeSet\times\stateSet)$, this projection is defined via
\begin{equation}\label{e:projection-def}
	\proj{m}{\mu} f := \argmin_{u \in \dictionarySpan} \
	\left\| u - f \right\|_{L^2_\mu}.
\end{equation}
It is well known that $\proj{m}{\mu}$ is a linear operator and satisfies $\|\proj{m}{\mu}f\|_{L^2_\mu}\leq \|f\|_{L^2_\mu}$. Note that the minimizer in \cref{e:projection-def} is unique as an element of $L^2_\mu$, but could be attained by multiple functions in $\dictionarySpan$ that agree $\mu$-almost everywhere. We denote the set of all minimizers in $\dictionarySpan$ by
\begin{equation*}
    \Argmin_{u \in \dictionarySpan} \ \left\| u - f \right\|_{L^2_\mu}
    :=
    \bigg\{ 
        \varphi \in \dictionarySpan: \;
        \left\| \varphi - f \right\|_{L^2_\mu} 
        = 
        \min_{u \in \dictionarySpan} \ \left\| u - f \right\|_{L^2_\mu} 
    \bigg\}.
\end{equation*}

\subsubsection{The infinite-data limit}\label{s:InfData}

The first step to establish the convergence of the approximate Lie derivative $\Lie_{mn}^\timestep \varphi$ is to consider the limit of infinite data ($n\to\infty$). For each fixed $n$, define the matrices
\begin{align*}
    A_n^\timestep &:= \frac{1}{n} \left( \Phi_n^\timestep \Psi_n^\top \right) = \frac{1}{n}\sum_{i=1}^n \vec{\phi}(t_i+\timestep,y_i) \vec{\psi}(t_i,x_i)^\top,
    \intertext{and}
    B_n &:= \frac{1}{n} \left( \Psi_n \Psi_n^\top \right) = \frac{1}{n}\sum_{i=1}^n \vec{\psi}(t_i,x_i) \vec{\psi}(t_i,x_i)^\top.
\end{align*}
The matrix $K_{mn}^\timestep$ used to define the approximate Lie derivative in \cref{e:Lmnh-def} satisfies
\begin{equation}\label{e:Kmnh-identity}
    K_{mn}^{\timestep} 
    =\left(\Phi_{n}^{\timestep}\Psi_{n}^{\top}\right)\left(\Psi_{n}\Psi_{n}^{\top}\right)^{\dagger}
    =A_{n}^{\timestep}B_{n}^{\dagger},
\end{equation}
so it is enough to identify the limits of $A_{n}^{\timestep}$ and $B_{n}^{\dagger}$ as $n$ tends to infinity.

The entries of $A_n^\timestep$ and $B_n$ in position $(i,j)$ are,  respectively, the empirical averages of the values of the functions $\phi_i(t+\tau,y)\psi_j(t,x)$ and $\psi_i(t,x)\psi_j(t,x)$ at the data snapshots. These functions are in $C_b(\timeSet\times\stateSet)$ by construction, so \cref{ass:weak-convergence} gives
\begin{subequations}\label{e:AB-convergence}
    \begin{align}
        \lim_{n\to\infty} A_n^{\timestep} &= \int\vec{\phi}(t+\timestep,y)\vec{\psi}(t,x)^{\top}\,d\rho(t,x,y)=: A^\timestep
        \\
        \lim_{n\to\infty} B_n &=\int\vec{\psi}(t,x)\vec{\psi}(t,x)^{\top}\,d\mu(t,x) =: B
        \label{e:Bn-convergence}
    \end{align}
\end{subequations}
almost surely. However, since pseudo-inversion is not a continuous operation, it is not clear that $B_n^\dagger \to B^\dagger$ in \cref{e:Kmnh-identity}. The next lemma shows that this is true almost surely (see also \cite{Wanner2022} for the case in which $B$ is assumed to have full rank).

\begin{lemma}
\label[lemma]{lem:pseudoinversion}
Under \cref{ass:weak-convergence},
$B_{n}^{\dagger}\to B^{\dagger}$ almost surely as $n\to\infty$.
\end{lemma}

\begin{proof}
Since pseudo-inversion is continuous along constant-rank sequences~\cite{Stewart1969}, it suffices to prove that $\rank(B_{n})=\rank(B)$ almost surely for sufficiently large $n$.
On the one hand, we almost surely have $\rank(B_{n})\leq\rank(B)$ when $n$ is large enough because $\ker(B)\subset \ker(B_n)$ almost surely. Indeed, if $\vec{v} \in \ker(B)$ then
\begin{equation*}
	0 = \vec{v}^\top B \vec{v}
	= \int \left|\vec{v} \cdot \vec{\psi}(t,x)\right|^2\,d\mu(t,x),
\end{equation*}
so $\vec{v} \cdot \vec{\psi}(t,x) = 0$ almost everywhere on the support of $\mu$. This means $\vec{v} \cdot \vec{\psi}(t_i,x_i) = 0$ almost surely for each data snapshot, so almost surely
\begin{equation*}
    B_n \vec{v} 
    = \frac{1}{n} \left( \Psi_n \Psi_n^\top \right) \vec{v} 
    = \frac{1}{n} \sum_{i=1}^n \vec{\psi}(t_i,x_i)\vec{\psi}(t_i,x_i)^\top \vec{v}
    = 0.
\end{equation*}
    
On the other hand, setting $r=\rank(B)$, we have that $\rank(B_n)\geq r$ almost surely for large enough $n$ because $\ker(B_n)$ does not contain the orthonormal eigenvectors $\vec{v}_1,\ldots,\vec{v}_r$ of $B$ corresponding to positive eigenvalues $\lambda_1\geq \lambda_2 \geq \cdots \geq \lambda_r>0$.
Indeed, by \cref{e:Bn-convergence} there almost surely exists $n_0 \in \mathbb{N}$ such that
$\|B_n - B\|_F \leq \tfrac12 \lambda_r$
when $n\geq n_0$. For every $j \in \{1,\ldots,r\}$, therefore,
\begin{equation*}
	\vec{v}_j^\top B_n \vec{v}_j
	\geq
	\vec{v}_j^\top B \vec{v}_j
	- \left\|B_n - B \right\|_F \|\vec{v}_j \|^2
	\geq
	\lambda_j	- \tfrac12 \lambda_r
    \geq \tfrac12 \lambda_r
	> 0.
    \qedhere
\end{equation*}
\end{proof}

We are now ready to prove that, as $n\to\infty$, the EDMD approximations $\koopman_{mn}^{\timestep}\varphi$ and $\Lie_{mn}^{\timestep}\varphi$ defined for a function $\varphi=\vec{c}\cdot\vec{\phi}$ in \cref{e:approx-koopman,e:Lmnh-def} converge, respectively, to 
\begin{subequations}
    \begin{align}
        \koopman_{m}^{\timestep}\varphi 
        &:=\vec{c}\cdot A^{\timestep}B^{\dagger}\vec{\psi},
        \\
        \Lie_{m}^{\timestep}\varphi 
        &:=\vec{c}\cdot \timestep^{-1}\left(A^{\timestep}B^{\dagger} - \Theta_m\right)\vec{\psi}.
    \end{align}
\end{subequations}
Moreover, these limits are $L_{\mu}^{2}$-orthogonal projections onto $\dictionarySpan$ of $\koopman^{\timestep}\varphi$ and of the difference quotient
\begin{equation*}
    \Lie^{\timestep}\varphi := \frac{\koopman^\timestep\varphi - \varphi}{\tau}.
\end{equation*} 
This result, stated precisely in \cref{thm:inf-data-limit} below, generalizes analogous statements for discrete-time processes \cite{eDMDconv,Crnjaric-Zic2020,Wanner2022} to a broader class of Markov stochastic processes. Moreover, our proof does not assume the matrix $B$ to be invertible, so $L_{\mu}^{2}$-orthogonal projections onto $\dictionarySpan$ are not uniquely defined as elements of $\dictionarySpan$ in general. This assumption was already dropped in \cite{Klus2020} in the context of stochastic differential equations, but with no justification of why $B_n^\dagger$ converges to $B^\dagger$.

\enlargethispage*{\baselineskip}
\begin{theorem}\label{thm:inf-data-limit}
    If \cref{ass:dictionary,ass:weak-convergence} are satisfied, then for any $\varphi\in \afSpan$ and any norm on $\dictionarySpan$ there almost surely holds
    \begin{align*}
        \koopman_{mn}^{\timestep}\varphi & \xrightarrow{n\to\infty}\koopman_{m}^{\timestep}\varphi, 
        \\
        \Lie_{mn}^{\timestep}\varphi & \xrightarrow{n\to\infty}\Lie_{m}^{\timestep}\varphi. 
    \end{align*}
    \vspace*{-\abovedisplayskip}
    Moreover,
    \vspace*{-\abovedisplayshortskip}
        \begin{align*}
            \koopman_{m}^{\timestep}\varphi & \in\Argmin_{u\in\dictionarySpan}\|u-\koopman^{\timestep}\varphi\|_{L_{\mu}^{2}},
            \\
            \Lie_{m}^{\timestep}\varphi & \in\Argmin_{u\in\dictionarySpan}\|u-\Lie^{\timestep}\varphi\|_{L_{\mu}^{2}}.
        \end{align*}
\end{theorem}

\begin{proof}
    The almost-sure convergence of $\koopman_{mn}^{\timestep}\varphi$ to $\koopman_{m}^{\timestep}\varphi$ follows upon applying \cref{e:AB-convergence} and \cref{lem:pseudoinversion} to \cref{e:Kmnh-identity}. Then, by \cref{e:Lmnh-def}, $\Lie_{mn}^{\timestep}\varphi$  converges almost surely to $\Lie_{m}^{\timestep}\varphi$.
   
    That $\koopman_{m}^{\timestep}\varphi$ is an $L_{\mu}^{2}$-orthogonal projections of $\koopman^{\timestep}\varphi$ is a standard calculation (see, e.g. \cite[Theorem~1]{eDMDconv}). To obtain the equivalent result for $\Lie_{m}^{\timestep}\varphi$, observe that
    \begin{equation*}
        \timestep^{-1}\left(A^{\timestep}B^{\dagger} - \Theta_m\right) = 
        \timestep^{-1}\left(A^{\timestep} - \Theta_m B\right) B^{\dagger}
        + \timestep^{-1}\Theta_m \left(B B^{\dagger} - I\right)
    \end{equation*}
    and that
    \begin{align}
        \timestep^{-1}\left(A^{\timestep} - \Theta_m B\right)
        &= \frac{1}{\timestep}\int \left[ \vec{\phi}(t+\timestep,y)\vec{\psi}(t,x)^{\top} - \Theta_m\vec{\psi}(t,x)\vec{\psi}(t,x)^{\top} \right]\,d\rho(t,x,y)
        \nonumber\\
        &=
        \int \frac{ \vec{\phi}(t+\timestep,y) - \vec{\phi}(t,x) }{\timestep} \, \vec{\psi}(t,x)^{\top}\,d\rho(t,x,y)
        \nonumber\\ 
        &=
        \int\frac{\koopman^\timestep\vec{\phi}(t,x) - \vec{\phi}(t,x)}{\timestep}  \, \vec{\psi}(t,x)^{\top}\,d\mu(t,x)
        \nonumber\\
        &=
        \int \Lie^\timestep\!\vec{\phi}(t,x) \, \vec{\psi}(t,x)^{\top}\,d\mu(t,x).
        \label{e:D-matrix-identity}
    \end{align}
    Set
    \begin{subequations}
        \begin{align}
            \label{e:xi-func-def}
            \xi_\varphi(t,x):= \vec{c}\cdot \timestep^{-1}\left(A^{\timestep} - \Theta_m B\right)\vec{\psi}(t,x),
            \\
            \eta_\varphi(t,x) := \vec{c}\cdot \timestep^{-1}\Theta_{m}(BB^{\dagger}-I)\vec{\psi}(t,x),
            \label{e:eta-func-def}
        \end{align}
    \end{subequations}
    so $\Lie_{m}^{\timestep}\varphi = \xi_\varphi + \eta_\varphi$.
    A calculation similar to that in \cite[Theorem~1]{eDMDconv} reveals that 
    \begin{equation*}
        \xi_\varphi \in \Argmin_{u\in\dictionarySpan}\|u-\Lie^{\timestep}\varphi\|_{L_{\mu}^{2}},
    \end{equation*}
    so we only need to show that $\eta_\varphi$ vanishes almost everywhere on $\supp\mu$. For this, note that $B$ is positive semidefinite, so 
    $B=V_{+}E_{+}V_{+}^{\top}$where $E_{+}$ is the diagonal matrix of positive eigenvalues and $V_{+}$ is the corresponding matrix of orthonormal eigenvectors. Writing $V_{-}$ for the matrix of eigenvectors of $B$ with zero eigenvalue, we find that $BB^{\dagger}-I= - V_{-}V_{-}^{\top}$ is a projection onto the kernel of $B$. This space is orthogonal to $\vec{\psi}(t,x)$ for almost every $(t,x)\in\supp\mu$ since
    \begin{equation*}
        \int\|V_{-}^{\top}\vec{\psi}\|^{2}\,d\mu=
        \trace \left[ V_{-}^{\top}\left(\int\vec{\psi}\vec{\psi^{\top}}\,d\mu\right)V_{-} \right]=
        \trace\left[V_{-}^{\top}BV_{-} \right]=0.
    \end{equation*}
    We conclude that $(BB^{\dagger}-I)\vec{\psi}$, hence $\eta_\varphi$,
    vanishes almost everyewhere on $\supp\mu$.
\end{proof}

\subsubsection{The infinite-sampling-rate limit}\label{s:InfSample}

We now turn to the limit of infinite sampling rate, when $\timestep\to0$.
Recall from \cref{ss:orth-proj} that $\mathcal{P}_{m}^{\mu}$, the $L_{\mu}^{2}$-orthogonal projection operator onto $\dictionarySpan$, is linear and satisfies $\|\mathcal{P}_{m}^{\mu}f\|_{L_{\mu}^{2}}\leq\|f\|_{L_{\mu}^{2}}$. Then, we deduce from \cref{thm:inf-data-limit} that 
\begin{align}
    \lim_{\timestep\to0}
    \left\|
        \Lie_{m}^{\timestep}\varphi 
        - \mathcal{P}_{m}^{\mu}\Lie\varphi
    \right\|_{L_{\mu}^{2}} 
    &=
    \lim_{\timestep\to0}
    \left\| 
        \mathcal{P}_{m}^{\mu}\Lie^{\timestep}\varphi
        - \mathcal{P}_{m}^{\mu}\Lie\varphi
    \right\|_{L_{\mu}^{2}}
    \nonumber \\
    &\leq
    \lim_{\timestep\to0}
    \left\| \Lie^{\timestep}\varphi - \Lie\varphi \right\|_{L_{\mu}^{2}}
    \nonumber \\
    &=0.
    \label{e:zero-timestep-initial}
\end{align}
The last equality follows because, by definition, $\Lie^{\timestep}\varphi$ converges uniformly to $\Lie\varphi$ on the whole space $\timeSet\times\mathbb{X}$, hence in particular on the support of the measure $\mu$. 

These simple steps show that, as $\timestep\to0$, the function $\Lie_{m}^{\timestep}\varphi$ converges in $L_{\mu}^{2}$ to a projection of the exact Lie derivative $\Lie\varphi$ onto $\dictionarySpan$. However, as remarked at the end of \cref{ss:orth-proj}, this projection need not be unique as an element of $\dictionarySpan$. It is therefore natural to ask whether $\Lie_{m}^{\timestep}\varphi$ converges to a \emph{particular} projection of $\Lie\varphi$. To answer this question, one must study the pointwise behaviour of $\Lie_{m}^{\timestep}\varphi$. 

As a candidate for the pointwise limit of $\Lie_{m}^{\timestep}\varphi$, we take the function $\mathcal{G}_m \varphi$ defined for every $\varphi=\vec{c}\cdot\vec{\phi}$ in $\afSpan$ by
\begin{equation}\label{e:Gm-def}
    \mathcal{G}_m \varphi(t,x) := \vec{c}\cdot C B^\dagger \vec{\psi}(t,x),
\end{equation}
where
\begin{equation*}
    C := \int \Lie\vec{\phi}(t,x) \, \vec{\psi}(t,x)^\top\,d\mu(t,x) \in \R^{\ell \times m}.
\end{equation*}
This choice is motivated by a connection with approximate Lie derivatives built using generator EDMD \cite{Klus2020}, which will be discussed in more detailed in \cref{ss:gEDMD}. A calculation similar to that in \cite[Theorem~1]{eDMDconv} confirms that $\mathcal{G}_m \varphi$ is indeed an $L^2_\mu$-orthogonal projection of $\Lie\varphi$.

\begin{lemma}\label[lemma]{lem:Gm-projection}
    For every $\varphi \in \afSpan \cap \mathcal{D}(\Lie)$, there holds
    \begin{equation*}
        \mathcal{G}_m \varphi \in \Argmin_{u\in\dictionarySpan}\|u-\Lie\varphi\|_{L_{\mu}^{2}}.
    \end{equation*}
\end{lemma}

This statement and the inequalities in \cref{e:zero-timestep-initial} imply that $\Lie_{m}^{\timestep}\varphi$ converges to $\mathcal{G}_m \varphi$ pointwise almost everywhere on the support of the data sampling measure $\mu$. However, from \cref{thm:inf-data-limit} we can actually deduce the following more precise result.

\begin{theorem}
\label{thm:h-to-zero}
    Let $\varphi\in\afSpan$ be represented as $\varphi=\vec{c}\cdot\vec{\phi}$ for $\vec{c}\in\mathbb{R}^{\ell}$. If \cref{ass:dictionary} is satisfied, then for every $(t,x) \in\timeSet \times \stateSet$ there holds
    \begin{equation*}
        \lim_{\timestep\to0}\Lie_{m}^{\timestep}\varphi(t,x)=
        \begin{cases}
            \mathcal{G}_{m}\varphi(t,x)
            &\text{if }\vec{c}\cdot\Theta_{m}(BB^{\dagger}-I)\vec{\psi}(t,x)=0,
            \\
            \infty 
            &\text{otherwise.}
    \end{cases}
    \end{equation*}
\end{theorem}

\begin{remark}
    This result includes the anticipated $\mu$-almost-everywhere convergence of $\Lie_{m}^{\timestep}\varphi$ to $\mathcal{G}_m \varphi$ as a particular case because, as shown in the proof of \cref{thm:inf-data-limit}, the function $\eta_\varphi = \vec{c}\cdot\Theta_{m}(BB^{\dagger}-I)\vec{\psi}$ vanishes almost everywhere on the support of $\mu$.
\end{remark}

\begin{proof}[Proof of \cref{thm:h-to-zero}]
    As in the proof of \cref{thm:inf-data-limit}, write $\Lie_{m}^{\timestep}\varphi = \xi_\varphi + \eta_\varphi$ where the functions $\xi_\varphi$ and $\eta_\varphi$ are defined in \cref{e:xi-func-def,e:eta-func-def}.
    Since $\Lie^{\timestep}\varphi \to \Lie\varphi$ uniformly on $\timeSet\times\stateSet$ as $\timestep\to0$ by definition of the Lie derivative, we can let $\timestep\to 0$ in \cref{e:D-matrix-identity} to find that
    \begin{equation*}
        \timestep^{-1}\left( A^\tau - \Theta_m B\right) = 
        \int (\Lie^\timestep\!\vec{\phi}) \, \vec{\psi}^{\top}\,d\mu
        \xrightarrow{\timestep \to 0}
        \int (\Lie\vec{\phi}) \, \vec{\psi}^{\top}\,d\mu = C.
    \end{equation*}
    We then conclude from \cref{e:xi-func-def} that $\xi_\varphi \to \mathcal{G}_m \varphi$ pointwise as $\timestep \to 0$. Thus, for every $(t,x)\in\timeSet\times\stateSet$, 
    \begin{align*}
        \lim_{\timestep\to0}\Lie_{m}^{\timestep}\varphi(t,x)
        &= 
        \lim_{\timestep\to0}\left[ \xi_\varphi(t,x)+ \eta_\varphi(t,x) \right]
        = 
        \mathcal{G}_m\varphi(t,x) 
        + 
        \lim_{\timestep\to0} \eta_\varphi(t,x)
        .
    \end{align*}
    The last limit is finite if and only if $\eta_\varphi(t,x) = \vec{c}\cdot\Theta_{m}(BB^{\dagger}-I)\vec{\psi}(t,x)=0$.
\end{proof}

An immediate consequence of \cref{thm:h-to-zero} is that $\Lie_{m}^{\timestep}\varphi\to\mathcal{G}_{m}\varphi$ pointwise as $\timestep \to 0$ if $B=\int\vec{\psi}\vec{\psi}^{\top}\,d\mu$ is invertible. This is true if and only if the following condition is met. 

\begin{assumption}\label{ass:psi-mu-incompatibility}
    If $u\in\dictionarySpan$ vanishes $\mu$-almost-everywhere, then $u\equiv 0$.
\end{assumption}

\begin{corollary}
    \label[corollary]{cor:h-to-zero-pointwise-conv}
    Under \cref{ass:dictionary,ass:psi-mu-incompatibility},
    $\Lie_{m}^{\timestep}\varphi\to\mathcal{G}_{m}\varphi$
    pointwise on $\timeSet\times\mathbb{X}$. 
\end{corollary}

\Cref{ass:psi-mu-incompatibility} is common in the EDMD literature (see also \cite[Assumption 1]{eDMDconv}) and holds, for instance, when
$\mathbb{X}=\mathbb{R}^{d}$, $\vec{\psi}$ is a
polynomial dictionary, and the support of the data sampling measure $\mu$ is not an algebraic set. This is true for the example involving the van der Pol oscillator in \cref{sec:vdp}, where pointwise convergence is indeed observed. A different oscillator example for which, instead, \cref{ass:psi-mu-incompatibility} and pointwise convergence fail is presented in \cref{sec:CircularOrbit}.

\subsubsection{The infinite EDMD dictionary limit}
\label{s:InfDictionary}

We finally turn to studying how the approximate Lie derivative $\Lie_{mn}^{\timestep}\varphi$ behaves as the approximation space $\dictionarySpan$ is enlarged. Precisely, we replace a fixed dictionary $\vec{\psi}$ with a sequence $\{\vec{\psi}^{m}\}_{m\geq \ell}$ of dictionaries of increasing size $m$. 

Our first (standard) result is that approximate Lie derivatives become increasingly accurate if the sequence $\{\vec{\psi}^{m}\}_{m\geq \ell}$ has the following approximation property.
\begin{assumption}\label{ass:approx-property}
    For every $u \in L^2_\mu(\timeSet\times\stateSet)$, there exists $u_m \in \dictionarySpan^{m}$ such that the sequence $\{u_m\}_{m\geq \ell}$ converges to $u$ in  $L^2_\mu$.
\end{assumption}
\noindent
Observe that this assumption does not require the inclusion $\dictionarySpan^{m} \subset \dictionarySpan^{m+1}$, even though this is often true in practice. This inclusion fails, for example, if $\{\vec{\psi}^{m}\}_{m\geq \ell}$ is a sequence of finite-element bases on increasingly fine but not nested meshes.
\begin{theorem}\label{thm:conv-m-infty-L2}
    Suppose \cref{ass:weak-convergence} holds and that the dictionaries $\{\vec{\psi}^{m}\}_{m\geq \ell}$ satisfy \cref{ass:dictionary,ass:approx-property}. Then,
	\begin{align*}
        \lim_{m\to\infty} \lim_{\timestep\to 0} \lim_{n\to\infty} \left\| \Lie^{\timestep}_{mn}\varphi - \Lie\varphi \right\|_{L^2_\mu}
        = \lim_{\timestep\to 0} \lim_{m\to\infty} \lim_{n\to\infty} \left\| \Lie^{\timestep}_{mn}\varphi - \Lie\varphi \right\|_{L^2_\mu}
        = 0.
	\end{align*}
	In particular, $\Lie^{\timestep}_{mn}\varphi(t,x)\to\Lie\varphi(t,x)$ for $\mu$-almost-every $(t,x)\in \timeSet\times\stateSet$.
\end{theorem}

\begin{proof}
    Recall from \cref{thm:inf-data-limit} that $\Lie^{\timestep}_m\varphi = \proj{m}{\mu}\Lie^\timestep\varphi$. Recall also from \cref{ss:orth-proj} that $\proj{m}{\mu}$ is a linear operator such that $\|\proj{m}{\mu}f\|_{L^2_\mu} \leq \|f\|_{L^2_\mu}$ and $\|\proj{m}{\mu}f-f\|_{L^2_\mu} \leq \|u-f\|_{L^2_\mu}$ for every $f$ and $u$. Given functions $u_m \in \dictionarySpan^{m}$ converging to $\Lie\varphi$ in ${L^2_\mu}$, which exist by assumption, we can therefore use the triangle inequality to estimate
    \begin{align*}
        \left\| \Lie^{\timestep}_{mn}\varphi - \Lie\varphi \right\|_{L^2_\mu}
		&\leq
		\left\| \Lie^{\timestep}_{mn}\varphi - \Lie^{\timestep}_m\varphi \right\|_{L^2_\mu}
		+ \left\| \Lie^{\timestep}_m\varphi - \proj{m}{\mu}\Lie\varphi \right\|_{L^2_\mu}
		+ \left\| \proj{m}{\mu}\Lie\varphi - \Lie\varphi \right\|_{L^2_\mu}.
        \\
		&\leq
		\left\| \Lie^{\timestep}_{mn}\varphi - \Lie^{\timestep}_m\varphi \right\|_{L^2_\mu}
		+ \left\| \Lie^{\timestep}\varphi - \Lie\varphi \right\|_{L^2_\mu}
		+ \left\| u_m - \Lie\varphi \right\|_{L^2_\mu}.
	\end{align*}
    The first term on the right-hand side vanishes as $n\to\infty$ by \cref{thm:inf-data-limit}. The other two terms vanish as $\timestep\to0$ and $m\to\infty$ because, by definition, $\Lie^{\timestep} \varphi \to \Lie \varphi$ uniformly and $u_m \to \Lie\varphi$ in $L^2_\mu$. These two limits can clearly be taken in any order.
\end{proof}

It would of course be desirable to complement \cref{thm:conv-m-infty-L2} with explicit convergence rates, but we do not pursue this here because the answer depends on the particular choices for the dictionaries $\vec{\phi}$, $\vec{\psi}^m$ and for the data sampling strategy (see also the discussion after \cref{ass:weak-convergence}). Interested readers can find an example of what can be achieved in~\cite{Zhang2022}, which estimates convergence rates for the EDMD-based identification of deterministic continuous-time systems.
Instead, to fully justify the good performance of approximate Lie derivatives in the examples of \cref{sec:FirstEx}, we study in more detail the special case in which every $\varphi\in\afSpan\cap\mathcal{D}(\Lie)$ satisfies $\Lie\varphi\in\dictionarySpan^m$ for all large enough $m$. This assumption is usually hard to verify in practice. When it holds, however, one recovers $\mathcal{L\varphi}$ pointwise on the full space as long as the dictionaries $\vec{\psi}^m$ also satisfy \cref{ass:psi-mu-incompatibility}.

\begin{theorem}\label{thm:conv-m-infty-poitwise}
    Suppose \cref{ass:weak-convergence} holds and that the dictionaries $\vec{\phi}$ and $\vec{\psi}^m$ satisfy \cref{ass:dictionary} for all $m$.
    Suppose also there exists $m_0\geq\ell$ such that, for every $m\geq m_0$:
    \begin{enumerate}[a{\rm)}, noitemsep, topsep=0.5ex]
        \item\label{m-conv-1}
        $\Lie\varphi\in\dictionarySpan^m$ for every $\varphi\in\afSpan$.
        \item\label{m-conv-2}
        If $u \in \dictionarySpan^{m}$ vanishes $\mu$-almost-everywhere, then $u\equiv0$.
    \end{enumerate}
    Then, for every $\varphi\in\afSpan$ 
    and every $m \geq m_0$,
    \begin{equation*}
        \lim_{\timestep\to 0}\lim_{n\to\infty} \Lie^{\timestep}_{mn}\varphi = 
        \Lie\varphi\qquad \text{pointwise on }\timeSet\times\mathbb{X}.
    \end{equation*}
\end{theorem}

\begin{proof}
    \Cref{thm:inf-data-limit,cor:h-to-zero-pointwise-conv} guarantee that $\Lie^{\timestep}_{mn}\varphi$ converges to the function $\mathcal{G}_m\varphi$ pointwise on $\timeSet\times\mathbb{X}$ as $n\to\infty$ and $\timestep\to0$. By \cref{lem:Gm-projection}, $\mathcal{G}_m\varphi$ minimizes $\|u-\Lie\varphi\|_{L^2_\mu}$ over all $u\in\dictionarySpan^{m}$. Then, since $\Lie\varphi\in\dictionarySpan^{m}$ by assumption \ref{m-conv-1}, we must have $\mathcal{G}_m\varphi=\Lie\varphi$ on the support of $\mu$. This implies $\mathcal{G}_m\varphi=\Lie\varphi$ on $\timeSet\times\mathbb{X}$ by assumption \ref{m-conv-2}.
\end{proof}

\subsection{Comparison to generator EDMD}
\label{ss:gEDMD}

The approximate Lie derivative operator $\Lie_{mn}^\timestep$ may be viewed as a difference quotient approximation for the generator of the (approximate) Koopman operator, where the timestep $\timestep > 0$ is determined by the rate at which data are sampled. If one can sample directly the Lie derivatives $\Lie\vec{\phi}$ of the elements in the dictionary $\vec{\phi}$, the difference quotient approximation can be avoided by using \textit{generator EDMD} (gEDMD) \cite{Klus2020}. Precisely, if the data snapshots $(t_{i},x_{i},y_{i})_{i=1}^{n}$ satisfy $y_{i}=\Lie\vec{\phi}(t_{i,}x_{i})$, one can build the data matrix
\begin{equation*}
    \Lambda_{n}:=\begin{bmatrix}| &  & |\\
    \Lie\vec{\phi}(t_{1},x_{1}) & \cdots & \Lie\vec{\phi}(t_{n},x_{n})\\
    | &  & |
    \end{bmatrix}
    \in \R^{\ell \times n}
\end{equation*}
and define an approximate Lie derivative operator $\mathcal{G}_{mn}:\afSpan\to\dictionarySpan$ by setting, for every $\varphi=\vec{c}\cdot\vec{\phi}$,
\begin{equation}
    \mathcal{G}_{mn}\varphi:=\vec{c}\cdot G_{mn}\vec{\psi}
    \qquad\text{where}\qquad
    G_{mn}:=\Lambda_{n}\Psi_{n}^{\dagger}.
    \label{e:Gmnh-def}
\end{equation}

One expects $\mathcal{G}_{mn}\varphi$ to approximate $\Lie\varphi$ because the $\ell\times m$ matrix $G_{mn}$ minimizes the least-squares error $\|\Lambda_{n}-G\Psi_{n}\|_{F}$. This expectation was justified theoretically in \cite{Klus2020} for stochastic processes governed by stochastic differential equations, and the results can be extended without much effort to the general class of Markov processes described in \cref{s:setup}. In particular, for the infinite-data limit we have the following analogue of \cref{thm:h-to-zero}, where $\mathcal{G}_m\varphi$ is defined as in \cref{e:Gm-def}.

\begin{theorem}\label{th:inf-data-gedmd}
    If \cref{ass:dictionary,ass:weak-convergence} holds then, for every $\varphi\in\afSpan$,
    \begin{equation*}
        \mathcal{G}_{mn}\varphi 
        \xrightarrow{n\to\infty}
        \mathcal{G}_{m}\varphi
    \end{equation*}
    almost surely.
    Moreover, $\mathcal{G}_{m}$ is an $L^2_\mu$-orthogonal projection of $\Lie\varphi$ onto $\dictionarySpan$.
\end{theorem}
\begin{proof}
    We already established in \cref{lem:Gm-projection} that $\mathcal{G}_{m}$ is an $L^2_\mu$-orthogonal projection of $\Lie\varphi$ onto $\dictionarySpan$. For the almost sure convergence, note that $G_{mn} = (\Lambda_{n}\Psi_{n}^\top)(\Psi_{n}\Psi_{n}^\top)^{\dagger}$. \Cref{lem:pseudoinversion} guarantees that $n(\Psi_{n}\Psi_{n}^\top)^{\dagger} = B_n^\dagger \to B^\dagger$ almost surely as $n\to\infty$. Similarly, by \cref{ass:weak-convergence} we almost surely have that
    \begin{equation*}
        \frac1n (\Lambda_{n}\Psi_{n}^\top) = \frac1n \sum_{i=1}^n \Lie\vec{\phi}(t_i,x_i) \vec{\psi}(t_i,x_i)^\top \xrightarrow{n\to\infty} \int \Lie\vec{\phi} \,\vec{\psi}^\top \,d\mu = C.
    \end{equation*}
    Thus, $\mathcal{G}_{mn}\varphi = \vec{c}\cdot G_{mn}\vec{\psi}$ converges to $\vec{c}\cdot CB^\dagger \vec{\psi} = \mathcal{G}_{m}\varphi$ almost surely.
\end{proof}

We also have the following analogues of \cref{thm:conv-m-infty-L2,thm:conv-m-infty-poitwise} for the case where $\vec{\psi}$ is chosen from a sequence of EDMD dictionaries $\{\vec{\psi}^m\}_{m\geq\ell}$ of increasing size.
\begin{theorem}
    Under the same assumptions as \cref{thm:conv-m-infty-L2}, for every $\varphi\in\afSpan$ there holds
    \begin{align*}
        \lim_{m\to\infty} \lim_{n\to\infty} \left\| \mathcal{G}_{mn}\varphi - \Lie\varphi \right\|_{L^2_\mu}
        = 0.
	\end{align*}
\end{theorem}
\begin{theorem}
    Under the same assumptions as \cref{thm:conv-m-infty-poitwise}, for every $\varphi\in\afSpan$ and every $m \geq m_0$, we have $\mathcal{G}_{mn}\varphi \to 
    \Lie\varphi$ pointwise on $\timeSet\times\stateSet$ as $n\to\infty$.
\end{theorem}

\Cref{th:inf-data-gedmd} can be combined with our results about the infinite sampling rate from \cref{s:InfSample} to discover that, in the limit of infinite data, approximate Lie derivatives constructed using EDMD with finite timesteo $\timestep$ do not generally reduce to the approximations constructed using gEDMD as $\timestep$ is decreased. Instead, they do so only under suitable conditions, such as \cref{ass:psi-mu-incompatibility}.
The same considerations carry over to the case of finite data if the underlying dynamics are \emph{deterministic}. In this case, a straightforward calculation shows that
\begin{equation*}
   \timestep^{-1}\left( A_n^\timestep - \Theta_m B_n \right) = \frac1n\sum_{i=1}^n \left[ \frac{\koopman^\timestep \vec{\phi}(t_i,x_i) - \vec{\phi}(t_i,x_i)}{\timestep} \right] \vec{\psi}(t_i,x_i)^\top
   \;
   \xrightarrow{\timestep \to 0} 
   \;\frac{1}{n} \Lambda_n \Psi_n^\top.
\end{equation*}
Then, since we can write
$L_{mn}^\timestep = \timestep^{-1} \left( A_n^\timestep - \Theta_m B_n \right)B_n^\dagger + \timestep^{-1} \Theta_m \left( B_nB_n^\dagger - I \right)$ and $G_{mn}= n^{-1} \Lambda_n \Psi_n^\top B_n^\dagger$, we find for every $\varphi=\vec{c}\cdot\vec{\phi}$ that
\begin{align*}
    \lim_{\timestep\to0}\Lie_{mn}^{\timestep}\varphi 
    & =
    \lim_{\timestep\to0}
    \left[\vec{c}\cdot \timestep^{-1} \left( A_n^\timestep - \Theta_m B_n \right)B_{n}^{\dagger}\vec{\psi}\right]
    +\lim_{\timestep\to0}
    \left[\vec{c}\cdot \timestep^{-1}\Theta_{m}(B_{n}B_{n}^{\dagger}-I)\vec{\psi}\right]
    \\
    & =
    \vec{c}\cdot \left(n^{-1} \Lambda_n \Psi_n^\top  B_{n}^{\dagger} \right) \vec{\psi}
    +\lim_{\timestep\to0}
    \left[\vec{c}\cdot \timestep^{-1}\Theta_{m}(B_{n}B_{n}^{\dagger}-I)\vec{\psi}\right]
    \\
    & =
    \mathcal{G}_{mn}\varphi
    +\lim_{\timestep\to0}
    \left[\vec{c}\cdot \timestep^{-1}\Theta_{m}(B_{n}B_{n}^{\dagger}-I)\vec{\psi}\right].
\end{align*}
Thus, the limit of $\Lie_{mn}^{\timestep}\varphi(t,x)$ exists and is equal to $\mathcal{G}_{mn}\varphi(t,x)$ only
at points $(t,x)$ that satisfy $\vec{c}\cdot\Theta_{m}(B_{n}B_{n}^{\dagger}-I)\vec{\psi}(t,x)=0$. This identity holds for the data points $(t_{i},x_{i})$, but not in general unless $B_n$ is invertible. This invertibility condition also does not hold in general, although it does so almost surely for large enough $n$ under \cref{ass:psi-mu-incompatibility}. Thus, contrary to what one might expect at first, we conclude that gEDMD and EDMD may produce very different Lie derivative approximations unless the EDMD dictionaries and the data sampling strategies satisfy suitable conditions. An example where gEDMD differs from the $\tau\to 0$ limit of EDMD is offered in \cref{sec:CircularOrbit}.

\section{Approximating auxiliary function from data}
\label{sec:sdps}

In \cref{sec:EDMD}, we constructed approximate Lie derivatives for functions in the span of an $\ell$-dimensional dictionary $\vec{\phi}$. Here, we show how these approximations can be used to construct approximate auxiliary functions via polynomial optimization. Motivated by the examples in \cref{sec:AuxFunctions}, we assume the auxiliary function $\varphi \in \afSpan$ one seeks must satisfy an inequality of the form 
\begin{equation}
	\label{e:poly-ineq-initial}
	a(t,x) \varphi(t,x) + b(t,x) \Lie \varphi(t,x) + c(t,x) \geq 0
	\qquad \forall (t,x) \in S \subseteq \timeSet \times \stateSet,
\end{equation}
where the functions $a$, $b$ and $c$ and the set $S$ are given.
We focus on the case of a single inequality for simplicity, but multiple inequalities can be handled in the same way. For clarity, we also restrict to the case where $\timeSet=\R$ and $\stateSet=\R^d$ and introduce the following assumptions.
\begin{assumption}\label{ass:poly-ineq}
	There exists an nonnegative integer $\omega$ such that
	\begin{enumerate}[a{\rm)}, noitemsep, topsep=0.5ex] 
		\item The functions $a$, $b$ and $c$ are fixed polynomials of degree at most $\omega$. 
		\item There exists polynomials $s_1(t,x),\,\ldots,\,s_k(t,x)$ of degree at most $\omega$ such that $$S = \{(t,x)\in\R\times\R^d:\; s_1(t,x)\geq 0,\,\ldots,\,s_k(t,x)\geq 0\}.$$
	\end{enumerate}
\end{assumption}

\begin{assumption}\label{ass:poly-dictionaries}
	The EDMD dictionaries $\vec{\phi}$ and $\vec{\psi}$ are polynomial dictionaries.
\end{assumption}

These restrictions ensure \cref{e:poly-ineq-initial} is a polynomial inequality when $\Lie\varphi$ is an explicitly computable polynomial or is replaced by its EDMD approximation $\Lie_{mn}^\timestep \varphi$ (which is a polynomial by \cref{ass:poly-dictionaries}). One can therefore handle this inequality with well-known tools for polynomial optimization, which are based on sum-of-squares polynomials and which we briefly review below for completeness (see \cite{Laurent2009sos,lasserre2015book,parrilo2013} for an in-depth treatment of the subject). Note, however, that our discussion can be generalized to the case in which the functions appearing in \cref{e:poly-ineq-initial} and in the definition of the set $S$ are \emph{semialgebraic}, meaning that their graphs are finite unions of sets defined by finitely many polynomial inequalities. Particular examples are sinusoidal functions and rational function with a fixed positive denominator. Finally, it is immediate to extend our discussion to include polynomials $c$ that depend affinely on tunable parameters, such as the constant $U$ in \cref{sec:ergodic-optimization}.

\subsection{A brief review of sum-of-squares techniques}

Given a vector $\xi=(\xi_1,\ldots,\xi_r)$, let $\R[\xi]$ be the vector space of $r$-variate polynomials with real coefficients and the entries of $\xi$ as the independent variables. For any integer $k$, let $\R[\xi]_k$ be the subspace of polynomials of degree $k$ or less. Observe that $\R[\xi]_k$ has finite dimension $\binom{r+k}{k}$. A polynomial $f \in \R[\xi]_{2k}$ is called a sum of squares (SOS) if there exists polynomials $g_1,\ldots,g_\gamma$ such that
\begin{equation*}
	f = g_1^2 + \cdots + g_\gamma^2.
\end{equation*}
The set of all SOS polynomials will be denoted by $\Sigma[\xi]$, while $\Sigma[\xi]_{2k}$ will be the subset of SOS polynomials of degree up to $2k$.

It is clear that SOS polynomials are nonnegative. The converse is true only for univariate polynomials ($r=1$), quadratic polynomials ($2k=2$), and bivariate quartic polynomials ($r=2$ and $2k=4$) \cite{hilbert1888}. On the other hand, while checking if a polynomial is nonnegative is NP-hard in general \cite{polyNP}, checking if it is SOS is a tractable \emph{semidefinite program} (SDP), that is, a convex optimization problem over positive semidefinite matrices constrained by linear equations. Indeed, let $L=\binom{r+k}{k}$ be the dimensions of the polynomial space  $\R[\xi]_k$ and fix a basis  $\vec{q}=(q_1,\ldots,q_L)$ for it. The following statement, where $\bS^L$ is the space of $L\times L$ symmetric matrices, follows from the (reduced) Cholesky factorization for positive semidefinite matrices.

\begin{lemma}[\cite{Parrilo2003,lasserre2001global,NesterovSOS}]
	A polynomial $f \in \R[\xi]_{2k}$ is SOS if and only if there exists $Q \in \bS^L$ positive semidefinite such that $f = \vec{q}^\top Q \vec{q}$.
\end{lemma}

\noindent
The equality $f = \vec{q}^\top Q \vec{q}$ yields linear constraints on $Q$ upon expanding both sides in a common basis for $\R[\xi]_{2k}$ and matching coefficients. These constraints remain linear if $f$ depends affinely on additional tunable variables, so optimizing polynomial coefficients subject to SOS constraints is also an SDP. Moreover, the formulation of this SDP can be done automatically by software toolboxes such as {\sc yalmip}~\cite{lofberg2004yalmip,lofberg2009pre}.

Finally, SOS polynomials enable one to formulate sufficient conditions for polynomial nonnegativity on sets defined by finitely many polynomial inequalities (such sets are called \emph{basic semialgebraic}). Specifically, given a basic semialgebraic set
\begin{equation*}
	S = \left\{ \xi\in\R^r: \; s_1(\xi)\geq 0,\,\ldots,\,s_k(\xi)\geq 0 \right\},
\end{equation*}
a sufficient condition for a polyonomial $f$ to be nonnegative on $S$ is that
\begin{equation}\label{e:wsos}
	f = \sigma_0 + \sum_{i=1}^k \sigma_k s_k
\end{equation}
for some SOS polynomials $\sigma_0,\ldots,\sigma_k$. We call this a \emph{weighted} SOS representation of $f$. Note that the degree of the SOS polynomials $\sigma_0,\ldots,\sigma_k$ is not generally known a priori, since one can often arrange for terms of degree larger than that of $f$ to cancel from the right-hand side. In practice, one checks \cref{e:wsos} with SOS polynomials $\sigma_0,\ldots,\sigma_k$ of chosen even degrees $2\omega_0,\ldots,2\omega_k$, so each $\sigma_i$ can be represented by a positive semidefinite matrix of size $\binom{r+\omega_i}{\omega_i}$. Then, \cref{e:wsos} leads to an SDP with $k+1$ positive semidefinite matrices constrained by affine equalities.

\subsection{Data-driven approximation of auxiliary functions}
\label{ss:data-sos}

The SOS techniques reviewed in the previous section can be utilized to construct approximate auxiliary functions that satisfy the approximate inequality
\begin{equation}\label{e:sos:data-ineq}
	a(t,x) \varphi(t,x) + b(t,x) \Lie^{\timestep}_{mn} \varphi(t,x) + c(t,x) \geq 0 
	\qquad \forall (t,x) \in S,
\end{equation}
which is obtained by replacing the exact Lie derivative in \cref{e:poly-ineq-initial} with its data-driven approximation. Indeed, under \cref{ass:poly-ineq,ass:poly-dictionaries}, for any $\varphi= \vec{c}\cdot \vec{\phi}$ in $\afSpan$ the left-hand side of \cref{e:sos:data-ineq} is a polynomial that depends affinely on the vector $\vec{c}$ (and, possibly, any tunable parameters appearing in the polynomial $c$). Thus, given nonnegative integers $\omega_0,\ldots,\omega_k$, we can optimize $\varphi$ by solving an SDP after strenghtening \cref{e:sos:data-ineq} into the weighted SOS condition
\begin{align*}
	a \varphi + b \Lie^{\timestep}_{mn} \varphi + c = \sigma_0 + \sum_{i=1}^k \sigma_i s_i,
\end{align*}
where $\sigma_i\in \Sigma[t,x]_{2\omega_i}$ for each $i\in\{0,\ldots,k\}$ and the polynomials $s_1,\ldots,s_k$ are those defining the set $S$ (cf. \cref{ass:poly-ineq}). (More precisely, one optimizes the vector $\vec{c}\in\R^\ell$ giving the representation of $\varphi$ in the polyomial dictionary $\vec{\phi}$.)

Note that one has considerable freedom to choose the half-degrees $\omega_0,\ldots,\omega_k$. For the computational examples in \cref{sec:FirstEx}, we always use the largest values such that the degree of $\sigma_0 + \sum_{i=1}^k \sigma_i s_i$ does not exceed that of $a \varphi + b \Lie^{\timestep}_{mn} \varphi + c$.

\begin{remark}\label{rem:convergence-rates}
The above approach enables one to construct an auxiliary function $\varphi$ that satisfies the approximate inequality \cref{e:sos:data-ineq}. However, one would like $\varphi$ to satisfy the original inequality \cref{e:poly-ineq-initial}. While it is straightforward to estimate 
\begin{equation*}
	a(t,x) \varphi(t,x) + b(t,x) \Lie \varphi(t,x) + c(t,x) \geq - \left\|b  \right\|_{L^\infty(S)} \left\| \Lie \varphi - \Lie^{\timestep}_{mn} \varphi \right\|_{L^\infty(S)},
\end{equation*}
it is not immediate to obtain explicit estimates for $\|\Lie \varphi - \Lie^{\timestep}_{mn} \varphi \|_{L^\infty(S)}$ as a function of the number of data snapshots $n$, the EDMD dictionary size $m$, and the timestep $\tau$. The main challenge is that our convergence analysis from \cref{sec:Theory} only guarantees that $\Lie^{\timestep}_{mn} \varphi$ converges to $\Lie \varphi$ in $L^2_\mu(\timeSet\times \stateSet)$. This is generally not enough to provide control on $\|\Lie \varphi - \Lie^{\timestep}_{mn} \varphi \|_{L^\infty(S)}$, even if one is willing to assume that $S$ is inside the support of the data sampling measure $\mu$ (note however, that this is not the case for the examples in \cref{sec:FirstEx}). Progress could be made if one further assumes that $\Lie \varphi$ belongs to a finite dimensional space, because then $\Lie \varphi - \Lie^{\timestep}_{mn} \varphi$ belongs to a (possibly different) space of finite dimension $N=N(m)$ for which the `inverse estimate' 
\begin{equation}\label{eq:inverse-estimate}
	\|\Lie \varphi - \Lie^{\timestep}_{mn} \varphi \|_{L^\infty(\timeSet\times \stateSet)} \leq C(N) \|\Lie \varphi - \Lie^{\timestep}_{mn} \varphi \|_{L^2_\mu}
\end{equation}
is available. We leave this to a future investigation. Here, we observe only that one should not expect quantitative estimates for $\|\Lie \varphi - \Lie^{\timestep}_{mn} \varphi \|_{L^\infty(S)}$ to be monotonic in the parameters $n$, $m$ and $\tau$. For example, the constant $C(N)$ must increase with $N$, which in turn increases as $m$ is raised. Thus, the right-hand side of \cref{eq:inverse-estimate} will generally not decrease monotonically in $m$, even if one takes $n\to \infty$ first.
\end{remark}
\section{Numerical examples}\label{sec:FirstEx} 
We now illustrate the construction of approximate auxiliary functions in five examples. The first one discovers a Lyapunov function from data (see \cite{moyalan2022data} for similar examples). The next two examples tackle ergodic optimization problems for deterministic and stochastic dynamics. The fourth example estimates pointwise bounds on a chaotic attractor. The last example is analytical and explicitly describes what happens when \cref{ass:psi-mu-incompatibility} does not hold. In particular, this example illustrates the which differences can arise if EDMD is replaced by gEDMD. In all computational examples, we construct polynomial auxiliary functions using {\sc yalmip}~\cite{lofberg2004yalmip,lofberg2009pre} and {\sc mosek} \cite{mosek2015mosek}. We use {\sc chebfun} \cite{ChebFun} to implement Chebyshev polynomials. Code to reproduce our results, and to experiment with smaller datasets or larger sampling times than those reported below, is available at \url{https://github.com/DCN-FAU-AvH/eDMD-sos}.

\subsection{Lyapunov functions}
	The two-dimensional map
	\begin{equation}\label{MGLyap}
		\begin{split}
			X_{t+1} &= \tfrac{3}{10}X_t, \\
			Y_{t+1} &= -Y_t + \tfrac{1}{2}Y_t + \tfrac{7}{18}X_t^2
		\end{split}
	\end{equation}
	has a globally asymptotically stable equilibrium at the origin. We seek to prove this by finding a Lyapunov function $V(x,y)$ satisfying
	\begin{subequations}\label{e:lyap:example:conditions}
		\begin{align}
			\label{e:lyap:example:positivity}
			V(x,y) - \varepsilon (x^2 + y^2) \geq 0,\\
			-\mathcal{L}V(x,y) - \varepsilon (x^2 + y^2) \geq 0,
			\label{e:lyap:example:decay}
		\end{align}
	\end{subequations}
	for some hyperparameter $\varepsilon>0$. These conditions imply that~\cref{e:lyap:positivity}--\cref{e:lyap:coercivity} hold with strict inequality away from the origin, as required to establish asymptotic stability. Note that one can always fix $\epsilon=1$ because one can always rescale $V$ by $\varepsilon$.

	To look for $V$ using our data-driven approach, we sampled the map~\cref{MGLyap} at $n=10^4$ uniformly distributed random points in the square $[-2,2]\times[-2,2]$. We then implemented the two inequalities in \cref{e:lyap:example:conditions} with $\varepsilon=1$ and with $\Lie V$ replaced by its data-driven approximation $\Lie_{mn}^\timestep V$ from \cref{sec:EDMD}. We used the weighted SOS approach of \cref{ss:data-sos} to search for polynomial $V$ of degree $4$, so $\vec{\phi}$ lists the $\ell=15$ monomials in $(x,y)$ of degree up to $4$, and we took all $m=45$ monomials of degree up to $8$ as the EDMD dictionary $\vec{\psi}$.  This choice ensures that $\Lie V \in \dictionarySpan$, but similar results are obtained when $\vec{\psi}$ includes also monomials of higher degree.

	Minimizing the $\ell^1$ norm of the coefficients of $V$ returns
	\begin{equation*}
		V(x,y) = 3.0815 x^2-1.5686 xy + 1.3333 y^2 -1.3038 x^3 + 0.5428 x^2y +0.2226 x^4,
	\end{equation*}
	where numerical coefficients have been rounded to four decimal places.
	Of course, this is only an approximate Lyapunov function: its positivity is guaranteed, as we have imposed \cref{e:lyap:example:positivity} exactly, but we do not know if its exact Lie derivative,
	\begin{equation*}
		\mathcal{L}V(x,y) = V\!\left(\tfrac{3}{10}x,-x + \tfrac{1}{2}y + \tfrac{7}{18}x^2\right) - V(x,y),
	\end{equation*}
	really satisfies \cref{e:lyap:example:decay} for some $\varepsilon>0$. This can be verified by maximizing $\varepsilon$ subject to~\cref{e:lyap:example:decay} for the given $V$. Doing so returns $\varepsilon \approx 0.9999$, so we have indeed constructed a Lyapunov function for the system.

	\begin{remark}
		The particular quartic $V$ constructed in this example has the special property that $\Lie V$ is also quartic. This means our data-driven approach gives the same answer when $\vec{\psi}$ lists only the monomials of degree up to 4, i.e., in the special case where $\afSpan = \dictionarySpan$. This is not true in general: in the following examples, the strict inclusion $\afSpan \subset \dictionarySpan$ is necessary to obtain accurate auxiliary functions.
	\end{remark}

	\begin{remark}\label{rmk:Lyapunov}
		Our data-driven discovery of Lyapunov functions is similar, but not equivalent, to the Koopman operator methods from \cite{deka2022koopman,zheng2022data}, which are a data-driven version of \cite{mauroy2016global,vaidya2008lyapunov}. Both approaches construct SOS Lyapunov functions in the form 
		\begin{equation*}
			V(x) = \frac{1}{2}\sum_{j = 1}^J \alpha_j|\vec{w}_j\cdot \vec{v}(x)|^2, \qquad \alpha_j \geq 0, 
		\end{equation*}
		where $\vec{v}$ is a basis of polynomials of degree $d\geq1$. In our approach we find such a $V(x)$ by tuning the vector $\vec{c}$ such that $V(x) = \vec{c}\cdot\vec{\phi}(x)$ admits an SOS decomposition and the approximate Lie derivative of $V$ is the negative of an SOS polynomial. Alternatively, \cite{deka2022koopman,zheng2022data} approximates the exact Lie derivative  
		\begin{equation*}
			\Lie V(x) = \sum_{j = 1}^J \alpha_j[(\Lie \vec{v}(x))^T \vec{w}_j\vec{w}^*_j\vec{v}(x) + \vec{v}(x) \vec{w}_j\vec{w}^*_j(\Lie\vec{v}(x))]   
		\end{equation*}
		by replacing $\Lie\vec{v}(x)$ with the vector $A\vec{v}$, where $A$ is a {\em square matrix} determined via EDMD. Having the $\vec{w}_j$ be left eigenvectors of the EDMD matrix associated to eigenvalues with negative real parts $\lambda_j$ leads to 
		\begin{equation*}
			\Lie V(x) \approx \sum_{j = 1}^J \mathrm{Re}(\lambda_j)\alpha_j|\vec{w}_j\cdot \vec{v}(x)|^2,   
		\end{equation*} 
		which is the negative of an SOS polynomial. While this has the advantage of producing an approximate Lyapunov function simply through an eigenvalue computation rather than the solution of an SOS problem, using Koopman eigenfunctions for other applications still requires the solution to SOS programs. In such cases, our approach offers the added flexibility of using two different EDMD polynomial dictionaries, which has potential for improving the accuracy of approximate Lie derivatives. We leave confirming this to future work.
	\end{remark}

\subsection{Ergodic optimization for the van der Pol oscillator}
\label{sec:vdp}
    Let us consider the van der Pol oscillator, given by the second-order ODE
    \begin{equation}\label{VdP}
        \ddot{X_t} - 0.1(1 - X_t^2)\dot{X}_t + X_t = 0.
    \end{equation}
    The state-space is $\stateSet = \mathbb{R}^2$, which corresponds to all possible values for $X_t$ and $\dot{X}_t$. We seek upper bounds on the long-time average of the `energy' of the system, here given by the observable
    \begin{equation}\label{VdPenergy}
        g(X_t,\dot{X}_t) = X_t^2 + \dot{X}_t^2.
    \end{equation}	 
    Equation \cref{VdP} has a stable limit cycle that attracts every initial condition except that at the unstable fixed point $(X_t,\dot{X}_t) = (0,0)$. This point saturates the trivial lower bound $g(X_t,\dot{X}_t)\geq 0$, while the long-time average of $g$ is maximized by the limit cycle.  

    \begin{table}[t] 
    \centering
    \caption{Data-driven upper bounds for the energy of the van der Pol oscillator \cref{VdP}, obtained with polynomial auxiliary functions of degree $\alpha$ and different integration times $T$ for the data collection. The final row gives bounds computed using the exact Lie derivative, \cref{VdPLie}, while the final column reports the average of the energy over the dataset collected for each integration time $T$.}
	\begin{tabular}{lccccccc}
		\toprule
		$T$ && $\alpha=4$ & $\alpha=6$ & $\alpha=8$ & $\alpha=10$ && Empirical Average\\
		\midrule
		$10^2$ && 6.1716 & 4.0100 & 4.0013 & 4.0011 && 2.2322  \\
		$10^{5/2}$ && 5.6799 & 4.0100 & 4.0013 & 4.0013 && 3.4418 \\
		$10^3$ && 5.3644 & 4.0100 & 4.0013 & 4.0010 && 3.8244   \\
		\midrule
		Exact && 6.6751 & 4.0100 & 4.0013 & 4.0012 && ---  \\
		\bottomrule
		\end{tabular}
    \label{tbl:VdP}
    \end{table}

    The goal of this example is to demonstrate that nearly sharp upper bounds can be established with less data than is required to observe convergence of a simple empirical average of the same data. For illustration, we generate synthetic data through numerical integration of the system \cref{VdP} with a timestep $\timestep = 0.001$, starting from the initial condition $(X_0,\dot{X}_0) = (0.1,0.2)$. Notice that the initial condition is chosen close to the unstable fixed point, meaning that there is an initial transient before falling into the stable limit cycle. This initial transient means that the long-time average of $g$ will take time to converge to its value along the limit cycle.

    \Cref{tbl:VdP} presents approximate upper bounds obtained by optimizing approximate auxiliary functions $V \in \afSpan$ with $\vec{\phi}$ listing all monomials in $(x,\dot x)$ up to degree $\alpha \geq 1$. The EDMD dictionary $\vec{\psi}$ was chosen to lists all monomials up to degree $\beta = \alpha + 2$. For each row of the table, data were collected by simulating \cref{VdP} up to the time horizon $T$ stated in the first column. The empirical average, obtained by averaging the energy observable \cref{VdPenergy} up to the given time horizon $T$, is presented in the final column. In the final row we provide the computed upper bound using the exact Lie derivative, here acting on differentiable functions $\varphi:\mathbb{R}^2 \to \mathbb{R}$ by
    \begin{equation}\label{VdPLie}
        \Lie \varphi(x,y) 
		= 
		\partial_x\varphi(x,y) y + \partial_y\varphi(x,y) [0.1(1 - x^2) y - x].
    \end{equation}  

	From both integrating \cref{VdP} far into the future on the limit cycle and the final row of \cref{tbl:VdP}, we find that the long-time average of the energy over the limit cycle is (to four significant digits) $4.001$. Notice that for all values of $T$ presented in the table the empirical average has not converged to this value, meaning that the initial transients are still influencing it. In contrast, if we approximate the Lie derivative with the same data and apply our data-driven bounding procedure, we are able to extract accurate approximate bounds even with the smallest dataset ($T = 10^2$). Thus, our data-driven approach enables us to extract system statistics from data long before they can be observed in the data itself.  
    
    \begin{figure}
        \centering
        \includegraphics[width=0.99\textwidth]{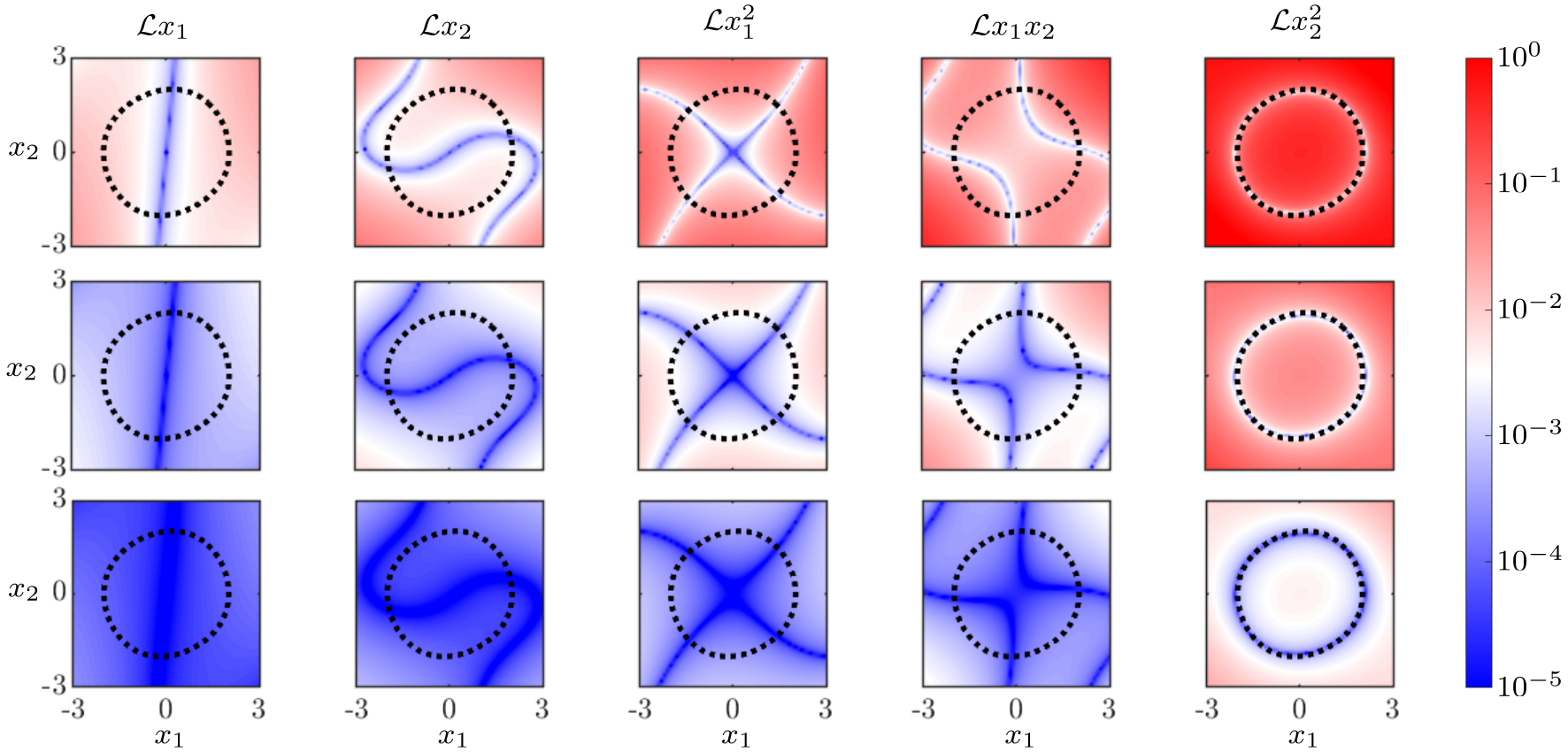}
        \caption{Error between $\Lie\varphi$ and its approximation $\Lie_{mn}^\timestep\varphi$ for $\varphi \in \{x_1,x_2,x_1^2,x_1x_2,x_2^2\}$ and timesteps $\timestep=10^{-2}$ (top row), $10^{-3}$ (middle row) and $10^{-4}$ (bottom row). In every case, $\vec{\psi}$ lists monomials of degree 4 and we sampled $n=10^6$ datapoints from the limit cycle (black dotted line).}
		\label{fig:vdp}
    \end{figure}

	The accuracy of the bounds in \cref{tbl:VdP} is due to an accurate approximation of the Lie derivative from data. Interestingly, since in this example the system dynamics are governed by a polynomial equation, we do not even require the transients to obtain such an accurate approximation. \Cref{fig:vdp} demonstrates that using only data sampled on the limit cycle we observe the {\em global} pointwise convergence of $\Lie_{mn}^\timestep \varphi$ to $\Lie \varphi$ on $\mathbb{R}^2$. Such convergence is a consequence of \cref{thm:conv-m-infty-poitwise}, whose second condition is satisfied because the limit cycle of the van der Pol oscillator is not an algebraic curve~\cite{Odani1995} and so cannot be contained in the zero level set of any element of an exclusively polynomial dictionary. The result is the ability to approximate the Lie derivative globally, rather than only in the region of state space where the data has been sampled from. Note, however, that this ability relies heavily on the dynamics being governed by polynomial equations and should not be expected in general.

\subsection{Ergodic optimization for a stochastic logistic map}
\label{ss:stoch-logistic-results}
The stochastic logistic map is given by
\begin{equation}\label{RandomLogistic}
	X_{t+1} = \lambda_t X_t(1 - X_t), \qquad t \in \mathbb{N},
\end{equation}
where $\lambda_t$ is drawn from the uniform distribution on $[0,4]$ for each $t \in \mathbb{N}$. The state-space is $\stateSet = \R$ and the unit interval $S=[0,1]$ is positively invariant. We seek to place upper and lower bounds on the long-time expected value of the observable ${g}(x) = x$. The auxiliary function framework for ergodic optimization in~\cref{sec:ergodic-optimization} applies to stochastic dynamics if one uses the stochastic definition of the Lie derivative. In our example, any auxiliary function $\varphi:\R \to \R$ has the stochastic Lie derivative
\begin{equation}\label{LogisticLie}
	\mathcal{L}\varphi = \frac{1}{4}\int_0^4 \varphi(\lambda x(1 - x))\mathrm{d}\lambda - \varphi(x).
\end{equation}

\begin{figure}[t]
	\centering
	\includegraphics[width=0.47\textwidth]{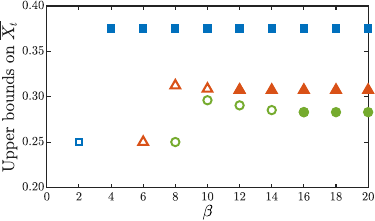}
	\hspace{10pt}
	\includegraphics[width=0.47\textwidth]{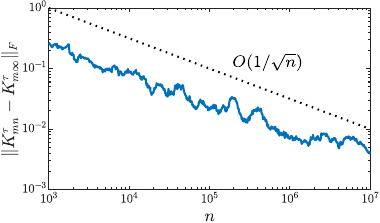}
	\caption{Left: Upper bounds on $\overline{X_t}$ for the the stochastic logistic map \cref{RandomLogistic}, obtained with EDMD dictionaries $\vec{\psi}$ of increasing size $m=\beta+1$ and polynomial auxiliary functions of degree $\alpha=2$ (squares), $6$ (triangles) and $8$ (circles). Symbols are full if $\beta\geq2\alpha$. Bounds are constant for $\beta\geq 2\alpha$.
	Right: Decay of the Frobenius norm $\|K_{mn}^\timestep - K_{m\infty}^\timestep\|_F$ with $n$, the number of data snapshots. Results are for $(\alpha,\beta) = (4,8)$, so $m = 9$, but are representative of other $(\alpha,\beta)$ combinations. 
	}
	\label{f:stoch-avg-results}
\end{figure}

We use our data-driven approach to construct approximate polynomial auxiliary functions of increasing degree $\alpha$. For numerical stability we represent polynomials using the Chebyshev basis $\vec{\phi} = (T_0(x),\ldots,T_\alpha(x))$ and we take $\vec{\psi}=(T_0(x),\ldots,T_{2\alpha}(x))$ as our EDMD dictionary. This choice ensures $\Lie\varphi \in \dictionarySpan$ for every $\varphi \in \afSpan$ but, as demonstrated by the left panel of \cref{f:stoch-avg-results}, the results do not change if one uses $\vec{\psi}=(T_0(x),\ldots,T_{\beta}(x))$ with $\beta\geq 2\alpha$.
We finally write $S = \{x\in\R:\, x-x^2\geq 0\}$, so \cref{ass:poly-ineq,ass:poly-dictionaries} are met.

Our dataset consists of one trajectory of the map with initial condition $x_0 \in (0,1)$ and $n = 10^7$ iterates, but we also implemented our approach using only the first $10^4$, $10^5$, and $10^6$ datapoints to investigate how results vary with $n$. Such a large amount of data is required to obtain accurate approximations of the Lie derivative for our stochastic map. Indeed, as shown in the right panel of \cref{f:stoch-avg-results}, the EDMD matrix $K_{mn}^\timestep$ converges at an $O(1/\smash{\sqrt{n}})$ rate to its infinite-data limit $K_{m\infty}^\timestep$, which can be calculated explicitly for this example using \cref{LogisticLie}.

\begin{table}[t] 
	\centering
	\caption{Data-driven `bounds' on $\overline{X_t}$ for the random logistic model~\cref{RandomLogistic}. Computations used degree-$\alpha$ polynomial auxiliary functions, an EDMD dictionary $\vec{\psi}$ listing monomials of degree up to $2\alpha$, and $n=10^4$--$10^7$ data snapshots. Exact bound values were computed using the exact Lie derivative~\cref{LogisticLie}.}
	\begin{tabular}{ r c ccccccc}
		\toprule
		&& $\alpha=2$ & $\alpha=4$ & $\alpha=6$ & $\alpha=8$ & $\alpha=10$ & $\alpha=12$ & $\alpha=14$  \\ [0.5ex]
		\midrule
		\parbox[t]{2ex}{\multirow{5}{*}{\rotatebox[origin=c]{90}{Upper Bound}}} &
		$n=10^4$
		& 0.3765 & 0.3162 & 0.3186 & 0.2844 & 0.2851 & 0.2858 & 0.2856 \\
		& $n=10^5$
		& 0.3751 & 0.3126 & 0.3086 & 0.2835 & 0.2814 & 0.2775 & 0.2757 \\
		& $n=10^6$
		& 0.3749 & 0.3124 & 0.3072 & 0.2832 & 0.2821 & 0.2758 & 0.2730 \\
		& $n=10^7$
		& 0.3751 & 0.3126 & 0.3070 & 0.2830 & 0.2817 & 0.2766 & 0.2737 \\
		\cmidrule{2-9}\\[-2ex]
		& Exact & 0.3750 & 0.3125 & 0.3069 & 0.2829 & 0.2816 & 0.2765 & 0.2736  \\
		\midrule
		\parbox[t]{2ex}{\multirow{5}{*}{\rotatebox[origin=c]{90}{Lower Bound}}} &
		$n=10^4$ & 0.0032 & 0.0055 & 0.0142 & 0.0107 & 0.0098 & 0.0090 & 0.0088 \\
		& $n=10^5$  & 0.0004 & 0.0016 & 0.0070 & 0.0057 & 0.0030 & 0.0248 & 0.0023 \\
		& $n=10^6$  & 0.0004 & 0.0010 & 0.0027 & 0.0059 & 0.0017 & 0.0032 & 0.0024 \\
		& $n=10^7$  & 0.0001 & 0.0001 & 0.0003 & 0.0011 & 0.0010 & 0.0016 & 0.0019 \\
		\cmidrule{2-9}\\[-2ex]
		& Exact  & 0.0000 & 0.0000 & 0.0000 & 0.0000 & 0.0000 & 0.0000 & 0.0000 \\
		\bottomrule
	\end{tabular}
	\label{tbl:Logistic}
\end{table}

The approximate upper and lower bounds on $\overline{X_t}$ we obtained are listed in \cref{tbl:Logistic} alongside exact bounds obtained with the exact Lie derivative \cref{LogisticLie}. This can be computed explicitly for polynomial $\varphi$ since the integral over $\lambda$ in \cref{LogisticLie} can easily be evaluated analytically. The data-driven `bounds' appear to converge to the exact ones in a non-monotonic fashion as $n$ increases, and the two agree to at least two decimal places for $n=10^7$. This confirms our approach works well with sufficient data.
The zero lower bound is sharp for \cref{RandomLogistic}, as it is saturated by the equilibrium trajectory $X_t=0$. Crucially, this trajectory is not part of our dataset, meaning that we learnt information about \emph{all} possible stationary distributions of the system even though we sampled data only from the \emph{single} stationary distribution approximated by the empirical distribution of the iterates $X_t$ in our simulated trajectory.
The upper bound, instead, decreases as $\alpha$ is raised and we conjecture it approaches the value $1/4$ of the stationary expectation of $X_t$, which we estimated by taking the average of our simulated trajectory. 
We also conjecture that the convergence with increasing $\alpha$ is slow because, for degree-$\alpha$ polynomial $\varphi$, the expression for $\Lie\varphi$ depends only on the first $\ell = \alpha+1$ moments of $\lambda$, which do not uniquely characterize its distribution. Thus, bounds obtained with fixed $\alpha$ apply to the \emph{maximum} stationary average of $X_t$, where the maximum is taken over all possible distributions of $\lambda$ whose moments of degree up to $\alpha$ coincide with those of the uniform distribution on $[0,4]$.

\subsection{Attractor bounds for a non-polynomial system}

As our next demonstration, we seek data-driven estimates on the maximum value of the state variables $x_1$, $x_2$ and $x_3$ on the chaotic attractor of the system 
\begin{equation}\label{Thomas}
\begin{split}
	\dot x_1 &= \sin(x_2) - 0.2x_1, \\
	\dot x_2 &= \sin(x_3) - 0.2x_2, \\
	\dot x_3 &= \sin(x_1) - 0.2x_3,
\end{split}
\end{equation}
introduced by Thomas \cite{thomas1999deterministic}. We generate synthetic data by integrating \cref{Thomas} with the initial condition $(x_1,x_2,x_3) = (0.1,0.2,0.3)$ and timestep $\tau = 0.001$. One can prove that the ball of radius 5 absorbs all initial conditions and, for numerical stability, we scale our simulation data via the linear transformation $\hat x = x/5$ to ensure the chaotic attractor is contained within the unit ball. 
We then use the pointwise bounding method of Section~\ref{sec:attractor-bounds} to obtain data-driven bounds on the functions $g_1(x) = x_1$, $g_2(x) = x_2$, and $g_3(x) = x_3$. Data-driven approximations of the constraints in \cref{AttractorBound} are implemented using the scaled variable $(\hat{x}_1, \hat{x}_2\hat{x}_3)$ and with $S = \{(\hat x_1,\hat x_2,\hat x_3):\ 1 - \hat x_1^2 - \hat x_2^2 - \hat x_3^2 \geq 0\}$. All results reported below, however, are for the original variables $(x_1,x_2,x_3)$.  We also fix $\lambda = 1$ in our computations, but direct the reader to \cite{goluskin2020bounding} for a discussion of the potentially complex dependence of the bounds on this parameter.   

\Cref{tbl:Thomas} presents the bounds we obtained when $\vec{\phi}$ is the set of all monomials in $(x_1,x_2,x_3)$ up to degree 4  and $\vec{\psi}$ is a monomial basis of increasing degree. Data was generated by integrating \cref{Thomas} up to $T = 10^2$, $10^3$, and $10^4$. Observe that although the cyclic symmetry of \cref{Thomas} implies that all pointwise bounds should be the same, this symmetry is not reflected in the training data. Thus, one should not expect the results of our data-driven implementation to be the same for all coordinates except in the infinite-data limit. This lack of symmetry is apparent in \cref{tbl:Thomas}, which however also demonstrates the improvement with increasing data: for integration time $T = 10^4$, bounds on the three coordinates $x_1$, $x_2$ and $x_3$ differ only on the order of the timestep.

\begin{table}[t] 
	\centering
	\caption{Data-driven estimates for the state-space variables $g_j(x) = x_j$, $j = 1,2,3$, on the attractor of the Thomas system \cref{Thomas}. Results are computed using the method of \cref{sec:attractor-bounds} with $\lambda=1$ and $V \in \linspan\vec{\phi}$, where $\vec{\phi}$ lists all monomials in $x$ of degree up to 4. Our data-driven implementation used $\vec{\psi}$ to list monomials in $x$ of increasing degree $\mathrm{deg}(\vec{\psi})$ and different integration times $T$.}
	\begin{tabular}{ r c ccccc}
		\toprule
		& $\mathrm{deg}(\vec{\psi})$ & $4$ & $5$ & $6$ & $7$ & $8$  \\ [0.5ex]
		\midrule
		\parbox[t]{2ex}{\multirow{2}{*}{\rotatebox[origin=c]{90}{$T = 10^2$}}} &
		$g_1(x) = x_1$
		& 3.8872 & 4.0886 & 3.9583 & 4.2445 & 4.2921 \\
		& $g_2(x) = x_2$
		& 3.4003 & 3.5564 & 3.7701 & 3.9552 & 3.9845 \\
		& $g_3(x) = x_3$
		& 3.9014 & 4.0591 & 4.0572 & 4.3643 & 4.2156 \\
		\midrule
		\parbox[t]{2ex}{\multirow{2}{*}{\rotatebox[origin=c]{90}{$T = 10^3$}}} &
		$g_1(x) = x_1$
		& 3.8445 & 3.9907 & 4.2413 & 4.4640 & 4.3309  \\
		& $g_2(x) = x_2$
		& 3.8146 & 3.9678 & 4.2123 & 4.4882 & 4.3887  \\
		& $g_3(x) = x_3$
		& 3.7918 & 3.9558 & 4.2395 & 4.4147 & 4.3467  \\
		\midrule
		\parbox[t]{2ex}{\multirow{2}{*}{\rotatebox[origin=c]{90}{$T = 10^4$}}} &
		$g_1(x) = x_1$
		& 3.8131 & 3.9650 & 4.2555 & 4.5115 & 4.4241  \\
		& $g_2(x) = x_2$
		& 3.8071 & 3.9658 & 4.2678 & 4.5266 & 4.4179  \\
		& $g_3(x) = x_3$
		& 3.7979 & 3.9602 & 4.2337 & 4.5021 & 4.4486   \\
		\bottomrule
	\end{tabular}
	\label{tbl:Thomas}
\end{table}

We must also stress that the numbers listed in \cref{tbl:Thomas} are not rigorous bounds on the maximal values that $x_1,x_2,x_3$ can attain on the attractor. Rather, they should be viewed as data-driven estimates. Indeed, the results for  $\mathrm{deg}(\vec{\phi}) = \mathrm{deg}(\vec{\psi}) = 4$ are not valid upper bounds, since they are smaller than the maximum value of $3.9564$ observed in a numerical simulation of \cref{Thomas} up to $T= 10^6$. While this may be undesirable, it is also to be expected as data-driven system analysis methods are inherently approximate. On the other hand, observe that the results improve dramatically if one moves beyond the typical setting of EDMD, where one uses $\vec{\phi} = \vec{\psi}$ to identify a square approximate Koopman matrix~\cite{eDMD}. This is because having $\afSpan\subset\dictionarySpan$, which is achieved using polynomial bases $\vec{\phi}$ and $\vec{\psi}$ of different degree, allows one to account for the nonlinearity of the underlying dynamics. In particular, one expects good results if the data is generated by an ODE whose vector field is well approximated by a polynomial of degree $\mathrm{deg}(\vec{\psi}) - \mathrm{deg}(\vec{\phi})$.  

Finally, observe that the approximate bounds in \cref{tbl:Thomas} are not monotonic as one increases either the integration time $T$ (hence, the amount of data) or the degree of the monomial basis $\vec{\psi}$.
The non-monotonicity in $T$ is due to the fact that our results are produced with a finite amount of data collected along one particular system trajectory, which as the integration time $T$ grows may spend different fractions of time on different parts of the attractor. For this reason, while \cref{thm:inf-data-limit} guarantees that data-driven approximate bounds for fixed bases $\vec{\phi}$ and $\vec{\psi}$ will converge to some (a priori unknown) value as $T \to \infty$, there is no reason to expect this convergence to be monotonic. Similarly, increasing the degree of $\vec{\psi}$ for fixed amount of data guarantees only a more accurate representation of the Lie derivatives of the basis function in $\vec{\phi}$ at the points in the training dataset, but does not rule out progressively worse approximations elsewhere. There are therefore no reasons to expect convergence at all when the training data is finite, let alone monotonic convergence. Moreover, for this example one cannot expect monotonic convergence even if one could pass to the infinite-data limit. Indeed, since the underlying dynamics are not polynomial, one cannot use Theorem~\ref{thm:conv-m-infty-poitwise} to ensure that Lie derivatives can be recovered from data arbitrarily accurately everywhere in the state space. In particular, while approximations at points on the attractor of \cref{Thomas} do improve as $\deg(\vec{\psi})$ is raised, approximations elsewhere could worsen. This is in stark contrast with the van der Pol example of \cref{sec:vdp}, where convergence on the attractor implies convergence on the whole state space.

\subsection{Ergodic optimization with a circular attractor}\label{sec:CircularOrbit}

We conclude with an example illustrating that if \cref{ass:psi-mu-incompatibility} does not hold, then Lie derivative approximations based on EDMD (\cref{sec:EDMD}) and gEDMD (\cref{ss:gEDMD}) can behave very differently from what one might expect. This behaviour is however consistent with the results proved in \cref{sec:Theory,ss:gEDMD}. In practice, therefore, one must be careful not to misinterpret results obtained with approximate auxiliary functions.

\paragraph{The problem} 
Consider the two-dimensional ODE
\begin{equation}\label{e:circular-system}
\begin{aligned}
	\dot{X_1} &= -X_2 + X_1(1-X_1^2-X_2^2)\\
	\dot{X_2} &= \phantom{-}X_1 + X_2(1-X_1^2-X_2^2),
\end{aligned}
\end{equation}
which has an unstable equilibrium point at $(x_1,x_2)=(0,0)$ and an attracting circular limit cycle $X_t=(\cos t, \sin t)$. We will use the auxiliary function framework of \cref{sec:ergodic-optimization} to find a lower bound $L$ on the time average of the quantity $g(x_1,x_2)=x_1^2 + x_2^2$. To make things concrete, we will look for a quadratic auxiliary function of the form 
\begin{equation}\label{e:circle:AF}
	V(x_1,x_2) = \gamma \left(1 + x_1^2 + x_2^2 \right),
\end{equation}
where $\gamma\in\R$ should be chosen such that the inequality
\begin{equation}\label{e:circle:bound-condition}
    x_1^2 + x_2^2 + \Lie V(x_1,x_2) - L \geq 0
\end{equation}
holds for all $x_1,x_2$ and the largest possible $L$. The exact Lie derivative is
\begin{equation}\label{e:circle:exact-Lie}
	\Lie V(x) = 2\gamma \left(x_1^2 + x_2^2\right) \left(1 - x_1^2 - x_2^2 \right)
\end{equation}
so with $\gamma=0$ we obtain the lower bound $L=0$. This lower bound is sharp, as it is saturated by the unstable equilibrium at the origin.

\paragraph{Data-driven lower bound via EDMD}
We now seek data-driven lower bounds when $\Lie V$ in \cref{e:circle:bound-condition} is replaced by its EDMD-based approximation $\Lie_{mn}^\timestep V$. We use $n$ data snapshots $(t_i, x_i, y_i)$ sampled at a rate $\timestep$ from the limit cycle, so
$t_i = i\timestep$, 
$x_i = (\cos t_i, \sin t_i)$ and $y_i = (\cos(t_i+\timestep), \sin(t_i+\timestep))$.
The function $V$ in \cref{e:circle:AF} belongs to the span of $\vec{\phi}=(1,x_1^2,x_2^2)$ and we use the particular EDMD dictionary $\vec{\psi}=(1,x_1^2,x_1x_2,x_2^2)$. Similar results are obtained with any dictionary $\vec{\psi} = (1,x_1^2,x_1x_2,x_2^2,\psi_5,\ldots,\psi_m)$ where $\psi_5,\ldots,\psi_m$ are monomials.

With these choices, the approximate Lie derivative $\Lie_{mn}^\timestep V$ can be calculated analytically using trigonometric identities for every $n$ and $\timestep$ to find
\begin{equation*}
	\Lie_{mn}^\timestep V(x_1,x_2) = \frac{\gamma}{3\timestep}\left( 1 - x_1^2 - x_2^2\right).
\end{equation*}
Thus, the approximate version of \cref{e:circle:bound-condition} with $\Lie V$ replaced by $\Lie_{mn}^\timestep V$ requires
\begin{equation*}
\frac{\gamma}{3\timestep}-L + \left(1 - \frac{\gamma}{3\timestep}\right)\left(x_1^2 + x_2^2\right) \geq 0\qquad \forall x_1,x_2.
\end{equation*}
Setting $\gamma = 3\timestep$ we find the lower bound $L=1$, which is evidently incorrect as it is violated by the equilibrium point at the origin.

This apparent contradiction can be explained by recalling from \cref{sec:ergodic-optimization} that a lower bound proved using the inequality $x_1^2 + x_2^2 + \Lie_{mn}^\timestep V(x_1,x_2)\geq L$ applies only to trajectories for which $\Lie_{mn}^\timestep V = \Lie V$, which is true only on the circles with radii $1$ and $1/(6\timestep)$. The system's limit cycle is the only trajectory remaining inside this set at all times, so the lower bound $L=1$ applies only to it (and is in fact sharp).

\paragraph{Data-driven lower bound via gEDMD}
We now repeat the exercise, but this time use the approximate Lie derivative $\mathcal{G}_{mn}V$ obtained with gEDMD as described in \cref{ss:gEDMD}. For this, we use data snapshots $\{(t_i, x_i, y_i)\}_{i=1}^n$ where $t_i=i\timestep$ and $x_i = (\cos t_i, \sin t_i)$ as before, but
\begin{equation*}
    y_i 
    = \Lie\vec{\phi}(x_i)
    = \begin{pmatrix}
    0 \\ -2 \cos t_i \sin t_i  \\ \phantom{-}2 \cos t_i \sin t_i
    \end{pmatrix}.
\end{equation*}

For our auxiliary function $V=\gamma(1+x_1^2+x_2^2)$ and dictionary $\vec{\psi}=(1,x_1^2,x_1x_2,x_2^2)$, one has $\mathcal{G}_{mn}V\equiv 0$ independently of $n$. The best lower bound provable with the inequality $x_1^2 + x_2^2 + \mathcal{G}_{mn} V(x_1,x_2)\geq L$ is therefore $L=0$, which is correct and sharp for all trajectories of \cref{e:circular-system}. Strictly speaking, however, this bound applies only to trajectories for which $\mathcal{G}_{mn}V = \Lie V$; it just so happens that these are exactly the unstable equilibrium and the limit cycle, which are the only invariant trajectories of the system.

\paragraph{Discussion}
In the examples above, the EDMD- and gEDMD-based Lie derivatives behave very differently when used to construct auxiliary functions. In particular, it is evident that $\Lie_{mn}^\timestep V\neq \mathcal{G}_{mn}V$, and none of these two functions recovers the exact Lie derivative \cref{e:circle:exact-Lie} on the full space. The same is true in the infinite-data limit ($n\to\infty$) because $\Lie_{mn}^\timestep V = \Lie_m^\timestep V$ and $\mathcal{G}_{mn}V =\mathcal{G}_m V$, as the left-hand sides are independent of $n$. Moreover, the function $\Lie_{m}^\timestep V$ converges to $\mathcal{G}_{m}V$ as $\timestep\to0$ only at points $(x_1,x_2)$ satisfying $x_1^2+x_2^2=1$.
This is exactly what \cref{thm:h-to-zero} predicts, since in our example we have 
\begin{equation*}
    \vec{c} = \begin{bmatrix}
        \gamma \\ \gamma \\ \gamma
    \end{bmatrix},
    \quad
    \vec{\psi} = \begin{bmatrix}
        1 \\ x_1^2 \\ x_1x_2 \\ x_2^2
    \end{bmatrix},
    \quad
    \Theta_m = \begin{bmatrix}
        1 & 0 & 0 & 0\\
        0 & 1 & 0 & 0\\
        0 & 0 & 0 & 1
    \end{bmatrix},
    \quad
    B = \frac{\pi}{4}\begin{bmatrix}
        8 & 2 & 0 & 2\\
        4 & 3 & 0 & 1\\
        0 & 0 & 1 & 0\\
        4 & 1 & 0 & 3
    \end{bmatrix},
\end{equation*}
%
giving $\vec{c}\cdot \Theta_m(BB^\dagger - I)\vec{\psi} = \gamma(1 - x_1^2 - x_2^2)$. Here, the matrix $B=\int \vec{\psi}\vec{\psi}^\top d\mu$ was computed by taking $\mu$ to be the uniform measure on the unit circle, which is the right choice for our data sampling strategy. 

Finally, we stress that the results in this example are very different to those obtained for the van der Pol oscillator in \cref{sec:vdp}, where $\Lie_{mn}^\timestep V$ converged to $\Lie V$ pointwise on $\R^2$ (cf. \cref{fig:vdp}). This could be anticipated because the limit cycle of \cref{e:circular-system} is an algebraic curve, meaning that it is the zero level set of a polynomial. Polynomial dictionaries $\vec{\psi}$ whose span includes polynomials in the form $p(x_1,x_2)(1-x_1^2-x_2^2)$ cannot therefore satisfy \cref{ass:psi-mu-incompatibility}. In contrast, the limit cycle of the van der Pol oscillator is not an algebraic curve \cite{Odani1995}, so \emph{any} polynomial dictionary $\vec{\psi}$ satisfies \cref{ass:psi-mu-incompatibility}. Therefore, rather remarkably, one is able to recover information about the global system dynamics even when sampling only on the limit cycle.
\section{Conclusion}\label{sec:Conclusion}

In this work we have provided a data-driven method for deducing information about dynamical systems without first discovering an explicit model. Our method combines two areas that are by now well-developed, namely, system analysis via auxiliary functions (sometimes also called Lyapunov or Lyapunov-like functions) and the data-driven approximation of the Koopman operator via EDMD. We also extended some known convergence results for EDMD to a broad class of stochastic systems, often under weaker assumptions than usual (cf. \cref{sec:Theory}). The result is a flexible and powerful method that can be applied equally easily to data generated by deterministic and stochastic dynamics, without any special pre-processing or other modifications to handle the stochasticity. Our examples have shown that we can accurately obtain Lyapunov functions from data, provide sharp upper bounds on long-time averages using less data than is required for an empirical average to converge, and bound expectations of stochastic processes. We expect a similar success when using auxiliary functions to study other properties of nonlinear systems. We also expect similar success when the data is polluted by small measurement noise, since the effect of the noise can be mitigated using filtering techniques \cite{nonomura2018,falconer2023} or integral formulations of model identification techniques that have straightforward extensions to EDMD \cite{WeakSINDy1,WeakSINDy2,McCalla}.

One potentially promising application of our method is as a pre-conditioner to discovering accurate and parsimonious dynamical models. For example, knowledge of Lyapunov functions, basins of attraction, or absorbing sets can improve data-driven model discovery from noisy or incomplete datasets \cite{ahmadi2020learning}. In particular, one can easily extend a variation of the SINDy method for constructing fluid flow models with an absorbing ball \cite{kaptanoglu2021promoting} to general systems with an absorbing set that need not be a ball: it suffices to first use our data-driven methods to identify a candidate absorbing set, and then construct a model that for which this set is indeed absorbing. Crucially, both steps can be implemented with convex optimization.

Although our theory does not put any limitations on the dimension of the data, both EDMD and the construction of auxiliary functions using semidefinite programming exhibit computational bottlenecks when the state-space dimension is not small. This can be seen clearly when the EDMD dictionaries are polynomial, since $\ell$ and $m$ grow considerably with the state-space dimension. Therefore, for even moderately-sized input data the resulting semidefinite programs could be prohibitively large. To overcome this issue in the setting of EDMD, \cite{KernelEDMD} proposes a kernel-based EDMD formulation that transfers one from estimating the Koopman operator with a matrix of size given by the large dictionary to learning one of size given by the number of snapshots $n$. This kernel formulation offers a significant computational speed-up in understanding the Koopman operator for systems such as discretized PDEs, where the state-space dimension is high and temporal data is difficult to produce. 
It is however not clear that similar techniques can help within our framework.  

There are also other potential avenues for future work. One is to establish convergence rates in the spirit of \cite{Zhang2022}. Although it is impossible to prove universal results in this direction \cite{Krengel1978}, one could hope to identify classes of systems and dictionaries for which convergence rates can be proved. Another interesting problem is to quantify the gaps between predictions made using data-driven auxiliary functions (e.g. bounds on time averages) and their rigorous model-based counterparts. Progress in this direction depends on whether the challenges outlined in \cref{rem:convergence-rates} can be resolved under realistic assumptions. Finally, it has recently been demonstrated that (approximate) Koopman eigenfunctions can be used directly to construct approximate Lyapunov functions without solving an SOS problem (see \cref{rmk:Lyapunov}). This relies on the fact that certain Koopman eigenfunctions capture stability properties of a system \cite{mauroy2016global,vaidya2008lyapunov}. It would be interesting to investigate whether the spectral analysis of the Koopman operator can shed light on other dynamical properties that have been studied via auxiliary functions.
In summary, many important questions remain to be answered and we believe this work only scratches the surface on what is possible at the intersection of Koopman theory, EDMD, and auxiliary function frameworks for system analysis.

\section*{Acknowledgments}
We are grateful for the hospitality of the University of Surrey during the 2022 `Data and Dynamics' workshop, where this work was started. We also thank Stefan Klus and Enrique Zuazua for their insight into EDMD. JB was partially supported by an Institute of Advanced Studies Fellowship at Surrey and an NSERC Discovery Grant.

\bibliographystyle{./siamplain}
\bibliography{./reflist.bib}
\end{document}